\documentclass{amsart}

\usepackage{amsmath,amsthm,amssymb}
\usepackage{mathtools}
\usepackage{tikz,tikz-3dplot,tikz-cd,graphicx}
\usepackage{standalone}
\usepackage{youngtab}
\usepackage[hidelinks]{hyperref}

\usetikzlibrary{shapes.geometric}

\graphicspath{{./}{figures/}}

\theoremstyle{plain}
\newtheorem{thm}{Theorem}[section]
\newtheorem{lem}[thm]{Lemma}
\newtheorem{cor}[thm]{Corollary}
\newtheorem{prop}[thm]{Proposition}

\newtheorem{mainthm}{Theorem}

\theoremstyle{definition}
\newtheorem{defn}[thm]{Definition}
\newtheorem{rem}[thm]{Remark}

\newtheorem{example}[thm]{Example}

\newcommand{\Z}{\mathbb{Z}}

\newcommand{\Ib}{\mathbf{I}}

\newcommand{\Sb}{\mathbf{S}}

\newcommand{\bb}{\mathbf{b}}

\newcommand{\Cc}{\mathcal{C}}

\newcommand{\Fc}{\mathcal{F}}
\newcommand{\cb}{\mathbf{c}}

\newcommand{\C}{\mathbb{C}}

\newcommand{\onto}{\twoheadrightarrow}

\newcommand{\cat}{\textsc{CAT}}

\newcommand{\NC}{\textsc{NC}}

\newcommand{\mult}{\textsc{Mult}}
\newcommand{\chains}{\textsc{Chains}}

\newcommand{\sym}{\textsc{Sym}}

\newcommand{\braid}{\textsc{Braid}}
\newcommand{\mon}{\textsc{Mon}}

\newcommand{\poly}{\textsc{Poly}}

\newcommand{\upper}{{\uparrow}}
\newcommand{\low}{{\downarrow}}
\newcommand{\bool}{\textsc{Bool}}

\newcommand{\rank}{\textrm{rk}}
\newcommand{\comp}{\textsc{Comp}}

\newcommand{\art}{\textsc{Art}}
\newcommand{\fact}{\textsc{Fact}}
\newcommand{\wcomp}{\textsc{WComp}}
\newcommand{\wfact}{\textsc{WFact}}

\newcommand{\bu}{\mathbf{u}}
\newcommand{\bv}{\mathbf{v}}
\newcommand{\bw}{\mathbf{w}}
\newcommand{\ba}{\textbf{a}}
\newcommand{\bx}{\mathbf{x}}
\newcommand{\by}{\textbf{y}}
\newcommand{\bz}{\textbf{z}}
\newcommand{\bp}{\textbf{p}}

\newcommand{\bs}{\mathbf{s}}

\newcommand{\NCS}[2]{\NC_{#1}^{#2}(\Sb)}

\newcommand{\closedint}{
  \begin{tikzpicture}[baseline]
    \draw[thick] (0,.1)--(.4,.1);
    \filldraw[color=black,fill=black!70!white] (0,.1) circle (.3mm)
    (.4,.1) circle (.3mm);
  \end{tikzpicture}
}

\tikzstyle{RowNode} = [
    shape = rectangle, 
    draw = black,
    fill = yellow!40, 
    rounded corners, 
    minimum height = 0.75cm, 
    minimum width = 1.5cm
]

\tikzstyle{BlueLine}=[line width=0.3mm,color=blue,text=black]
\tikzstyle{BluePoly}=[BlueLine,fill=blue!20]
\tikzstyle{RedLine}=[line width=0.3mm,color=red,text=black]
\tikzstyle{RedPoly}=[RedLine,fill=red!20]
\tikzstyle{GreenLine}=[thick,color=black!30!green,text=black]
\tikzstyle{GreenPoly}=[thick,color=green!50!black,fill=green!30,join=bevel]
\tikzstyle{PurpleLine}=[line width=0.3mm,color=red!60!blue!80,text=black]
\tikzstyle{PurplePoly}=[PurpleLine,fill=red!40!blue!50]
\tikzstyle{OrangeLine}=[thick,color=orange]
\tikzstyle{GrayLine}=[thick,color=black!50!gray]
\tikzstyle{GrayPoly}=[GrayLine,fill=gray!20]
\tikzstyle{dot}=[shape=circle,draw,color=black,fill=black,inner sep=1.5pt]
\tikzstyle{opendot}=[shape=circle,draw,color=black,fill=white,inner sep=1.5pt]
\tikzstyle{bigdot}=[dot,inner sep=2pt]
\tikzstyle{littledot}=[dot,inner sep=0.75pt]
\tikzstyle{littleopendot}=[opendot,inner sep=0.75pt]
\tikzstyle{smalldot}=[dot,inner sep=1pt]
\tikzstyle{disk}=[thick,shape=circle,draw,color=black,fill=yellow!10]
\tikzstyle{lightdisk}=[thick,shape=circle,draw,color=black!40,fill=yellow!20]
\tikzstyle{plate}=[thick,shape=rectangle,draw,color=black,fill=yellow!10,rounded
  corners,minimum size=1.1cm] 

\newcounter{joncomments}

\newcounter{michaelcomments}

\begin{document}

\title[Continuous Noncrossing Partitions]
      {Continuous Noncrossing Partitions and\\ Weighted Circular Factorizations}
\author{Michael Dougherty} 
    \email{doughemj@lafayette.edu}
    \address{Department of Mathematics, Lafayette College,
    Easton, PA 18042}
\author{Jon McCammond}
    \email{jon.mccammond@math.ucsb.edu}
    \address{Department of Mathematics, UC Santa Barbara, 
    Santa Barbara, CA 93106} 
\date{\today}

\begin{abstract}
  This article examines noncrossing partitions of the unit circle in
  the complex plane; we call these \emph{continuous noncrossing
  partitions}.  More precisely, we focus on the \emph{degree-$d$}
  continuous noncrossing partitions where unit complex numbers in the
  same block have identical $d$-th powers.  We prove that the
  degree-$d$ continuous noncrossing partitions form a topological
  poset whose uncountable set of elements can be indexed by
  equivalence classes of objects we call \emph{weighted linear
  factorizations} of factors of a $d$-cycle.  Moreover, the maximal
  elements in this poset form a subspace homeomorphic to the dual
  Garside classifying space for the $d$-strand braid group.

  The degree-$d$ continuous noncrossing partitions of the unit circle
  are a special case of a more general construction. For every choice
  of Coxeter element $\cb$ in any Coxeter group $W$ we define a
  topological poset of equivalence classes of weighted linear
  factorizations of factors of $\cb$ in $W$ whose elements we call
  \emph{continuous $\cb$-noncrossing partitions}.  The maximal
  elements in this poset form a subspace homeomorphic to the
  one-vertex complex whose fundamental group is the corresponding dual
  Artin group.
\end{abstract}

\maketitle

\section*{Introduction}

Let $\NCS{}{}$ denote the set of all partitions of the unit circle
$\Sb \subset \C$ for which the convex hulls of the blocks are pairwise
disjoint. This set is partially ordered by refinement, and we refer to
its elements as \emph{continuous noncrossing partitions}. The full
continuous noncrossing partition poset $\NCS{}{}$ is quite
complicated. For example, a geodesic lamination of a hyperbolic
surface lifts to a lamination of the disk model of the hyperbolic
plane, and this induces a continuous noncrossing partition of the
circle whose blocks are the asymptotic ends of the leaves.  In this
article we focus on the simpler subposet $\NCS{d}{}$ of
\emph{degree-$d$ continuous noncrossing partitions}, i.e. those where
unit complex numbers in the same block have the same image under the
map $z\mapsto z^d$.  See Figure~\ref{fig:ncs-example}.  Our first main
theorem gives an algebraic construction of $\NCS{d}{}$.

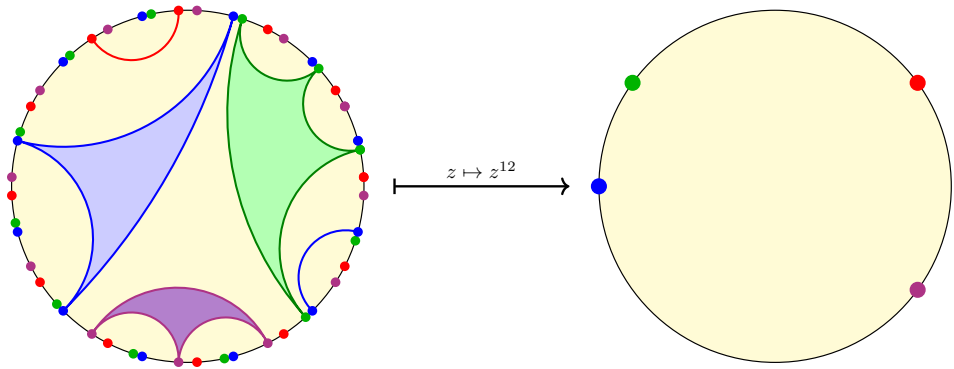
\begin{figure}
  \centering

\tikzstyle{dot}=[shape=circle,draw,color=black,fill=black,inner sep=1.3pt]

\begin{tikzpicture}
  \def\ra{3cm}
  \def\mid#1#2{(#2-#1)/2}
  \def\Arc#1#2#3{arc ({#1+270}:{#2+90}:{tan(\mid{#1}{#2})*#3})}
  \def\drawArc#1#2#3#4{\draw[#1] (#2:#4) \Arc{#2}{#3}{#4}}
  
  \begin{scope}[shift = {(-5,0)}, line width = 0.35mm]
    \filldraw[color=black,fill=yellow!20, line width = 0.2mm] (0,0) circle (\ra);
    
    \drawArc{color=red}{93}{123}{\ra}; 
    
    \filldraw[color=green!50!black,fill=green!30] (12:\ra) 
    \Arc{12}{42}{\ra} 
    \Arc{42}{72}{\ra}
    \Arc{72}{312}{\ra}
    \Arc{312}{372}{\ra};
    
    \drawArc{color=blue}{315}{345}{\ra};  
    
    \filldraw[color=blue,fill=blue!20] (75:\ra) 
    \Arc{75}{165}{\ra}
    \Arc{165}{225}{\ra}
    \Arc{225}{435}{\ra};
    
    \filldraw[color=red!60!blue!80,fill=red!40!blue!50] (237:\ra) 
    \Arc{237}{267}{\ra}
    \Arc{267}{297}{\ra}
    \Arc{297}{597}{\ra};
    
    \foreach \i in {0,...,12} {
      \node[dot,color=red] () at (\i*30+3:\ra) {};
      \node[dot,color=black!30!green] () at (\i*30+12:\ra) {};
      \node[dot,color=blue] () at (\i*30+15:\ra) {};
      \node[dot,color=red!60!blue!80] () at (\i*30+27:\ra) {};
    }
  \end{scope}

  \draw[|->, line width = 0.4mm] (-1.5,0) -- node [above] {$z\mapsto z^{12}$} (1.5,0);
  
  \begin{scope}[shift = {(5,0)}, line width = 0.25mm]
    \filldraw[color=black,fill=yellow!20, line width = 0.2mm] (0,0) circle (\ra);
    
    \node[dot,color=red, inner sep = 2.5pt] () at (36:\ra) {};
    
    \node[dot,color=black!30!green, inner sep = 2.5pt] () at (144:\ra) {};
    
    \node[dot,color=blue, inner sep = 2.5pt] () at (180:\ra) {};
    
    \node[dot,color=red!60!blue!80, inner sep = 2.5pt] () at (324:\ra) {};
    
  \end{scope}
\end{tikzpicture}

  \caption{The degree-$12$ continuous noncrossing partition shown on
    the left has five non-trivial blocks, and these are sent under the
    map $z \mapsto z^{12}$ to the four points shown on the right.  The
    $48$ points shown on the left are the preimages of these four
    points on the right.  All other points in $\Sb$ belong to
    singleton blocks.}
  \label{fig:ncs-example}
\end{figure}

\begin{mainthm}[Theorem~\ref{thm:degree-d-invariant-isom}]
  \label{mainthm:nc-fact-isomorphism}
  The poset $\NCS{d}{}$ of degree-$d$ continuous noncrossing
  partitions is isomorphic to the topological poset
  $\Fc(\sym_d,\delta,\Sb)$ defined using weighted linear
  factorizations of noncrossing permutations in $\sym_d$.
\end{mainthm}

The topological poset $\Fc(\sym_d,\delta,\Sb)$ is a special case of a
general construction introduced here.  Since the full definition is
slightly complicated, we sketch the defintion here and record the
details in the body of the article.  First, recall that several recent
studies of affine and hyperbolic Artin groups have been based on
constructions of the following form.  Given a group $G$ with a fixed
conjugacy-closed generating set $X$ and an element $g \in \mon(X)
\subset G$, the submonoid generated by $X$, one can define a bounded
graded poset $P_g=[1,g]$ of reduced factorizations of $g$ over $X$, a
simplicial complex $O_g = \Delta(P_g)$ which is the order complex of
$P_g$, a one-vertex $\Delta$-complex $K_g = q(O_g)$ which is a
cellular quotient of $O_g$, and a group $G_g = \pi_1(K_g)$ which is
the fundamental group of the quotient.  These are the \emph{interval
poset} $P_g$, its \emph{order complex} $O_g$, the \emph{interval
complex} $K_g$ and the \emph{interval group} $G_g$, respectively.  See
\cite{BradyMcCammond-fac-euc-isom, McCammond-dual-euc-art,
  McCammondSulway-euc-art, mccammond-euc-art-survey,
  paolinisalvetti21, haettel-survey, delucchi24}.

When $G$ is the symmetric group $\sym_d$, $X = T$ is the set of all
transpositions in $\sym_d$, and $\delta$ is the $d$-cycle
$(1\ 2\ \cdots\ d)$, the poset $P_\delta$ is the noncrossing partition
lattice $\NC_d$ and the interval complex $K_\delta$ is the \emph{dual
braid complex} with fundamental group $\braid_d$. This space was
introduced by Tom Brady \cite{brady01} and given a piecewise-Euclidean
metric by Brady and the second author \cite{BrMc-10}.  The dual braid
complex is a classifying space for the $d$-strand braid group and it
is conjectured to be locally $\cat(0)$ with its ``orthoscheme''
metric.  This claim has been verified for $d\leq 7$ \cite{BrMc-10,
  haettel-kielak-schwer-16, Jeong_2023}.  A proof that these metric
piecewise Euclidean complexes are nonpositively curved would resolve
the long-standing conjecture that the braid group is a $\cat(0)$
group.  More generally, when $G = W$ is a Coxeter group, $X=T$ is the
set of all reflections in $W$, and $g=\cb$ is a Coxeter element,
(i.e. a product of the reflections in a simple system in some order),
then $P_g$ is the poset of $\cb$-noncrossing partitions, $K_g$ is the
one-vertex complex for the dual Artin group and $G_g$ is the dual
Artin group $\art^\ast(G,g)$.

\begin{table}
  \begin{center}
    \begin{tabular}{|c||c|c|c|}
      \hline
      & Cells & Points & TopPoset\\
      \hline \hline
      General & $\fact(G,g,\Ib)$ & $\wfact(G,g,\Ib)$ &
      $\Fc(G,g,\Ib)$\\  
      \hline
      Coxeter & $\fact(W,\cb,\Ib)$ & $\wfact(W,\cb,\Ib)$ &
      $\Fc(W,\cb,\Ib)$\\ 
      \hline
      Symmetric & $\fact(\sym_d,\delta,\Ib)$ &
      $\wfact(\sym_d,\delta,\Ib)$ & $\Fc(\sym_d,\delta,\Ib)$\\ 
      \hline
      Integer &  $\comp(\Z,n,\Ib)$ & $\wcomp(\Z,n,\Ib)$ &
      $\Cc(\Z,n,\Ib)$\\ 
      \hline
    \end{tabular}
  \end{center}
  \caption{Linear factorizations of $g$ describe the cells of the
    order complex $O_g$ (first column). Weighted linear factorizations
    of $g$ describe the points of $O_g$ (second column).  The disjoint
    union of order complexes $\{O_h \mid h \leq g\}$ form a 
    topological poset with a grading (third column).\label{tbl:linear}}
\end{table}

In this article, we introduce algebraic constructions that index the
cells and the points of the order complex $O_g$ and interval complex
$K_g$.  The cells are indexed by linear and circular factorizations.
The points are indexed by weighted linear and weighted circular
factorizations.  The linear versions refer to the simplicial order
complex $O_g$. In symbols, $\textsc{Cells}(O_g) = \chains(P_g) =
\fact(G,g,\Ib)$, and $\textsc{Points}(O_g) = \textsc{WChains}(P_g) =
\wfact(G,g,\Ib)$.  The circular versions refer to the one-vertex
interval complex $K_g = q(O_g)$.  In symbols, we have
$\textsc{Cells}(K_g) = q(\textsc{Cells}(O_g)) = q(\chains(P_g)) =
q(\fact(G,g,\Ib)) = \fact(G,g,\Sb)$.  See Tables~\ref{tbl:linear}
and~\ref{tbl:circular}. The rows in these tables vary the type of
group under consideration, and in the final row we switch to additive
notation.

\begin{table}
  \begin{center}
    \begin{tabular}{|c||c|c|c|}
      \hline
      & Cells & Points & TopPoset\\
      \hline \hline
      General & $\fact(G,g,\Sb)$ & $\wfact(G,g,\Sb)$ &
      $\Fc(G,g,\Sb)$\\  
      \hline
      Coxeter & $\fact(W,\cb,\Sb)$ & $\wfact(W,\cb,\Sb)$ &
      $\Fc(W,\cb,\Sb)$\\ 
      \hline
      Symmetric & $\fact(\sym_d,\delta,\Sb)$ &
      $\wfact(\sym_d,\delta,\Sb)$ & $\Fc(\sym_d,\delta,\Sb)$\\ 
      \hline
      Integer & $\comp(\Z,n,\Sb)$ & $\wcomp(\Z,n,\Sb)$ &
      $\Cc(\Z,n,\Sb)$\\ 
      \hline
    \end{tabular}
  \end{center}
  \caption{Circular factorizations of $g$ describe the cells of the
    interval complex $K_g$ (first column).  Weighted circular
    factorizations of $g$ describe the points of $K_g$ (second
    column).  The circular topological poset (third column) is a
    quotient of the linear topological poset in the third column of
    Table~\ref{tbl:linear}. \label{tbl:circular}}
\end{table}

Our third construction combines the weighted factorizations of
elements below $g$ into a topological poset.  A \emph{topological
poset}, roughly speaking, is a poset where the elements have a
topology and the order relation is a closed subspace
(Definition~\ref{def:top-poset}).  They were defined by Rade
\v{Z}ivaljevi\'c in \cite{zivaljevic98} and inspired by the work of
Vasiliev in \cite{vasiliev91}. See also \cite{folkman66, bjorner95,
  zivaljevic16}.

One of the motivating examples for the concept is the poset
$\mathcal{L}_n(\mathbb{R})$ of linear subspaces in $\mathbb{R}^n$
partially ordered by inclusion.  It has a rank function that sends
each subspace to its dimension, and can be viewed topologically as the
disjoint union of the Grassmannians $\text{Gr}(i,\mathbb{R}^n)$ with
$i \in \{0,\ldots,n\}$.  A more relevant example here is the poset of
multisets of size at most $n$ in an interval $\Ib$, also ordered by
inclusion.  The elements of rank $k$ are the multisets of size $k$
when counted with multiplicity, and these are in bijection with the
points of a $k$-dimensional simplex, which gives them a natural
topology.  Figure~\ref{fig:top-graded-poset} in
Section~\ref{sec:top-posets} shows this topological graded poset for
$n=3$.  It has a tetrahedron of points in rank~$3$, a triangle of
points in rank~$2$, an interval of points in rank~$1$, and a single
point in rank~$0$.

The elements of the linear topological poset $\Fc(G,g,\Ib)$ are
precisely \emph{weighted circular factorizations} of elements $h$ in
the interval $P_g =[1,g]$, and topologically it is a disjoint union of
the order complexes $\{O_h \mid 1 \leq h \leq g\}$.  The circular
topological poset $\Fc(G,g,\Sb)$ is defined as a quotient of the
linear topological poset $\Fc(G,g,\Ib)$.  See
Definition~\ref{def:circular-factorization} for a precise definition
of the quotient map.  Example~\ref{ex:subfact-delta} shows that the
topology of $\Fc(G,g,\Sb)$ need not be a disjoint union of the
interval complexes $\{K_h \mid 1 \leq h \leq g\}$.  The components of
$\Fc(G,g,\Sb)$ are instead cyclic covers of the interval complexes
$K_h$.  The maximal elements of $\Fc(G,g,\Sb)$ form a copy of the
interval complex $K_g$ without taking a cyclic cover.

\begin{mainthm}[Theorem~\ref{thm:max-elements}]
  \label{mainthm:max-elements}
  Let $P_g$ be an interval poset in a marked group $G$ with order
  complex $O_g$ and interval complex $K_g$.  In the circular
  topological poset $\Fc(G,g,\Sb)$, the subspace of maximal elements
  is homeomorphic to $K_g$.
\end{mainthm}

When Theorem~\ref{mainthm:max-elements} is applied to a Coxeter
group $W$ with Coxeter element $\cb$, the result is a topological
poset whose maximal elements form the one-vertex complex for the dual
Artin group $\art^\ast(W,\cb)$.

\begin{mainthm}[Corollary~\ref{cor:dual-complex}]
  \label{mainthm:dual-complex}
  Let $\cb \in W$ be a Coxeter element in a Coxeter group generated by
  its reflections.  The maximal elements of the circular topological
  poset $\Fc(W,\cb,\Sb)$ form a subspace homeomorphic to the complex
  $K_\cb$ whose fundamental group is the dual Artin group
  $\art^\ast(W,\cb)$.  In particular, the subspace of maximal elements
  in $\NCS{d}{} \cong \Fc(\sym_d,\delta,\Sb)$ is homeomorphic to the
  dual braid complex $K_\delta$ whose fundamental group is the braid
  group $\braid_d$.
\end{mainthm}

The ideals and filters in the poset $\Fc(G,g,\Ib)$ (and in its
circular counterpart $\Fc(G,g,\Sb)$) are quite distinct.  Let $\bu \in
\Fc(G,g,\Ib)$ be a weighted linear factorization of $h \in P_g =
   [1,g]$ and let $[\bu]$ be its equivalence class in $\Fc(G,g,\Sb)$.
   The lower set (ideal) $\low([\bu])$ of all elements below $[\bu]$
   in the partial order on $\Fc(G,g,\Sb)$ is isomorphic to a product
   of intervals $[1,x_1] \times \cdots \times [1,x_k]$ for some
   $x_1,\ldots,x_k \in P_g$ (Remark~\ref{rem:lower-sets}). In
   particular, the lower set $\low([\bu])$ is discrete when $G$ is a
   discrete group and finite when $G$ is a finite group, even though
   $\Fc(G,g,\Sb)$ itself is uncountable.  The upper set
   $\upper([\bu])$ is more complicated: it is a circular topological
   poset for an element in $[1,g]$.

\begin{mainthm}[Corollary~\ref{cor:circle-upper-sets}]
  \label{mainthm:upper-sets}
  Let $\bu$ be a weighted linear factorization of $h \in P_g$ with
  equivalence class $[\bu] \in \Fc(G,g,\Sb)$. Then the upper set
  $\upper([\bu])$ in $\Fc(G,g,\Sb)$ is isomorphic to $\Fc(G,h',\Sb)$
  where $h\cdot h' = g$. Consequently, the maximal elements of
  $\upper([\bu])$ form a subspace which is isometric to the interval
  complex $K_{h'}$ inside the interval complex $K_g$.
\end{mainthm}

In the specific case of degree-$d$ continuous noncrossing partitions
$\NCS{d}{}$, there is a close connection with degree-$d$ polynomials
that has already been noted in the literature.  The preimage of the
union of the real and imaginary coordinate axes under a degree-$d$
complex polynomial, for example, determines a degree-$d$ continuous
noncrossing partition of $\Sb$ in which each non-trivial block has
size divisible by~$4$ \cite{martin-savitt-singer-07}.  The blocks are
determined by the points in the circle at infinity asymptotically
approached by the ends of a connected component of the preimage.
Additional constructions connecting continuous noncrossing partitions
and complex polynomials can be found in our earlier articles
\cite{gcp1, gcp2}.

The theorems proved here allow us to tackle the following natural
question: what is the relationship between the dual braid complex and
other classifying spaces for the braid group?  One classical example
is the space $\poly_d^{mc}(\mathbb{C}^*)$ of all monic complex
polynomials with $d$ roots which are centered at the origin and
critical values in the punctured plane $\mathbb{C}^*$; by observing
that a polynomial has distinct roots if and only if it does not have
$0$ as a critical value, this space can be identified with the
configuration space of $d$ unordered points in the plane which are
centered at the origin. Fox and Neuwirth showed in 1962 that
$\poly_d^{mc}(\mathbb{C}^*)$ is a classifying space for the $d$-strand
braid group, which means that this space is homotopy equivalent to the
dual braid complex for abstract reasons.  In this article we provide a
more concrete connection.

\begin{mainthm}[Corollary~\ref{cor:dual-braid-spine}]
  \label{mainthm:dual-braid-spine}
  The dual braid complex is a spine for $\poly_d^{mc}(\mathbb{C}^*)$.
  More specifically, the dual braid complex is homeomorphic to the
  subspace of polynomials with critical values on the unit circle, and
  there is a deformation retraction from $\poly_d^{mc}(\mathbb{C}^*)$
  to this subspace.
\end{mainthm}

The existence of the deformation retraction (without reference to a
cell structure) was stated by W. Thurston in an unpublished manuscript
from 2012, which was posthumously completed by Baik, Gao, Hubbard,
Lei, Lindsey, and D. Thurston \cite{thurston20}. In this 2012
manuscript, Thurston introduced \emph{degree-$d$-invariant
laminations}, which provide a useful tool for describing points in the
subspace of polynomials with critical values on the unit circle.  In
this article, we use the preceding theorems to describe a
characterization of points in the dual braid complex which aligns with
Thurston's degree-$d$-invariant laminations, thus providing an
explicit embedding of the dual braid complex into $\poly_d^{mc}(\mathbb{C}^*)$. 
To complete the proof of Theorem~\ref{mainthm:dual-braid-spine}, we invoke
Thurston's deformation retraction.

The connection between spaces of polynomials and the dual braid
complex was pointed out to us by Daan Krammer in 2017, and his comment
inspired much of the work in this article and other recent work by the
authors \cite{DoMc-20,gcp1,gcp2}.  In \cite{gcp2}, we describe a
bounded piecewise-Euclidean metric for $\poly_d^{mc}(\mathbb{C}^*)$
and a finite cell structure for its metric completion such that the
subspace $\poly_d^{mc}(\Sb)$ of polynomials with critical values on
the unit circle inherits a metric cell structure which is isometric to
the dual braid complex. In particular, \cite{gcp2} contains an
alternative direct proof of Theorem~\ref{mainthm:dual-braid-spine}
which does not rely on Thurston's argument.

A general version of Theorem~\ref{mainthm:dual-braid-spine} appears in
work of David Bessis \cite[Section~10]{bessis15}, where he uses a
complex of spaces on the universal cover of
$\poly_d^{mc}(\mathbb{C}^*)$ together with tools from algebraic
topology \cite[Section~4.G]{hatcher} to prove the result.  In fact,
his work proves a version of Theorem~\ref{mainthm:dual-braid-spine}
that extends to all finite complex reflection groups, including all
finite Coxeter groups.  Concretely, Bessis proves that if $\cb$ is a
Coxeter element for a finite Coxeter group $W$, then the interval
complex $K_\cb$ is a spine for the corresponding quotient of the
complexified hyperplane complement \cite{bessis15}. The constructions
introduced here might lead to an alternate algebraic proof of Bessis'
result.

\subsection*{Structure of the article} 
Sections~\ref{sec:posets} and \ref{sec:complexes}
review the definitions and properties of interval posets, orthoschemes
and interval complexes. Section~\ref{sec:factors} introduces linear
and circular factorization posets and examines their combinatorial
structure.  Section~\ref{sec:wt-factors} introduces weighted
analogs of these factorizations and examines their topology and
geometry.  Section~\ref{sec:top-posets} defines graded topological
posets of weighted factorizations and proves
Theorems~\ref{mainthm:max-elements},
\ref{mainthm:dual-complex}, and
\ref{mainthm:upper-sets}. Section~\ref{sec:cont-nc-part} defines the
poset $\NCS{d}{}$ of degree-$d$-invariant partitions of the circle and
proves Theorems~\ref{mainthm:nc-fact-isomorphism} and
\ref{mainthm:dual-braid-spine}.

\section{Posets}\label{sec:posets}

In this section we recall some background information for partially
ordered sets and a special kind of metric simplex known as an
orthoscheme.  See \cite[Ch.~3]{stanley-ec1} for a standard reference
on posets and \cite{hatcher,BrMc-10} for background on simplices.

A partially ordered set $P$ is a \emph{lattice} if each pair of
elements has a unique meet and a unique join. A poset is
\emph{bounded} if it has a unique maximum element and a unique minimum
element.  A \emph{chain} of length $k$ in $P$ is a collection of
distinct elements $x_0,\ldots,x_k \in P$ such that $x_0 < \cdots <
x_k$, and a chain is \emph{maximal} if it is not properly contained in
another chain.  We write $\chains(P)$ for the poset of all chains in
$P$ under inclusion.  We say that $P$ is \emph{graded} if there is a
\emph{rank function} $\rank \colon P \to \mathbb{N}$ such that for all
$x,y\in P$, we have $\rank(x) < \rank(y)$ whenever $x < y$ and
$\rank(x) + 1 = \rank(y)$ whenever $x<y$ and there is no $z\in P$ with
$x < z < y$. A finite bounded poset is graded if and only if all of
its maximal chains have the same length, which we call the
\emph{height} of the poset. Given two elements $x,y\in P$, the set of
all $z\in P$ with $x \leq z \leq y$ is the \emph{interval} $[x,y]$.
The set of all elements $y\in P$ with $x\leq y$ is called the
\emph{upper set} of $x$ and is denoted $\upper(x)$.  Similarly, the
set of all $z\in P$ with $z\leq x$ is the \emph{lower set} of $x$,
denoted $\low(x)$.

\begin{example}[Boolean lattice]
  The \emph{Boolean lattice} $\bool(n)$ consists of all subsets for
  the $n$-element set $\{1,\ldots,n\}$, partially ordered under
  inclusion, and it is indeed a lattice: given $A,B\in \bool(n)$, the
  unique meet is $A\cap B$ and the unique join is $A\cup B$. This
  poset is also graded, with rank function $\rank\colon\bool(n) \to
  \mathbb{N}$ given by $\rank(A) = |A|$ for each $A\subseteq
  \{1,\ldots,n\}$. Fixing an element $A\in \bool(n)$ with rank $k$,
  the lower set $\low(A)$ is isomorphic to the smaller Boolean lattice
  $\bool(k)$, whereas the upper set $\upper(A)$ is isomorphic to
  $\bool(n-k)$. Finally, we write $\bool^*(n)$ to mean the subposet of
  nonempty elements in $\bool(n)$ and refer to this as the
  \emph{truncated Boolean poset}.
\end{example}

There is a close connection between partially ordered sets and
simplicial complexes.  Each cell complex has an associated \emph{face
poset} and each poset has an associated \emph{order complex}. We will
make use of both operations, and we begin by describing the first.

\begin{defn}[Face posets]\label{def:face-poset}
  Let $X$ be a simplicial complex. The \emph{face poset} $P(X)$ is
  defined to be the graded poset of all faces of $X$ (including the
  empty face), ordered by inclusion.
\end{defn}

For example, each face of an $n$-dimensional simplex can be specified
precisely by a subset of the $n+1$ vertices, and the relation of
incidence between two faces corresponds exactly to inclusion between
the two subsets.  In other words, the face poset for the $n$-simplex
is isomorphic to the Boolean lattice $\bool(n+1)$.

\begin{defn}[Order complexes]\label{def:order-complex}
  Let $P$ be a graded poset. The \emph{order complex} $\Delta(P)$ is
  the ordered simplicial complex with vertex set $P$ and an ordered
  $k$-simplex on the vertices $x_0,\ldots,x_k$ whenever $x_0 < \cdots
  < x_k$ is a chain in $P$.  By construction, there is a bijection
  between $\textsc{Cells}(\Delta(P))$ and $\chains(P)$.  See Hatcher's
  book for background on ordered simplicial complexes and
  $\Delta$-complexes \cite{hatcher}.
\end{defn}

The order complex of $\bool(n)$, for example, is homeomorphic to the
cell complex obtained by subdividing the cube $[0,1]^n \subset
\mathbb{R}^n$ into $n!$ top-dimensional simplices via the
$\binom{n}{2}$ hyperplanes with equations $x_i = x_j$ where $i\neq j$.
Each of the top-dimensional simplices is determined by a path of
length $n$ from $(0,\ldots,0)$ to $(1,\ldots,1)$ along the edges of
the cube which gives an ordering of the vertices for each
simplex. Moreover, we can promote this homeomorphism to an isometry
with an appropriate choice of metric for the order complex.

\begin{defn}[Orthoschemes]\label{def:orthoscheme}
  The simplex spanned by points $\bp_0,\ldots,\bp_n$ in $\mathbb{R}^n$
  is called an \emph{$n$-dimensional orthoscheme} if the set of $n$
  vectors $\{\bp_{i} - \bp_{i-1} \mid i \in [n]\}$ is orthogonal. If
  those vectors are orthonormal, then the simplex is a \emph{standard
  $n$-dimensional orthoscheme}, and it is isometric to the subset of
  points $(x_1,\ldots,x_n) \in \mathbb{R}^n$ subject to the
  inequalities $0 \leq x_1 \leq x_2 \leq \cdots \leq x_n \leq 1$.
\end{defn}

For each bounded graded poset $P$, we give the order complex
$\Delta(P)$ the \emph{orthoscheme metric}, in which each maximal
simplex is a standard orthoscheme with the order of its vertices
determined by the order of the corresponding maximal chain in $P$. For
more detail on this use of the orthoscheme metric, see
\cite[Sec. 5-6]{BrMc-10}.

To conclude this section, we give a brief remark on the faces of
orthoschemes.

\begin{rem}\label{rem:orthoscheme-faces}
  Each face of a standard $n$-dimensional orthoscheme is itself an
  orthoscheme, but one which is not necessarily standard. Using the
  inequalities $0 \leq x_1 \leq x_2 \leq \cdots \leq x_n \leq 1$ to
  describe the standard $n$-dimensional orthoscheme, each nonempty $A
  = \{t_0,\ldots,t_k\}$ in $\bool(n+1)$ corresponds to the
  $k$-dimensional face determined by the equations
  \begin{align*}
    x_i = 0 &\text{ if } 1 \leq i \leq t_0; \\
    x_i = x_{i+1} &\text{ if } t_j < i < i+1 \leq t_{j+1}
    \text{ for some } j \in \{0,1,\ldots,k-1\}; \\
    x_i = 1 &\text{ if } t_{k} < i \leq n+1.
  \end{align*} 
  In plain language, each element of $\bool(n+1)$ determines a face by
  changing some of the $n+1$ inequalities into equalities. See
  Table~\ref{tbl:set-face} for an example when $n=2$.  Furthermore,
  this face is isometric to the set of points $(y_1,\ldots,y_k) \in
  \mathbb{R}^k$ determined by the inequalities
  \[
  0 \leq \frac{y_1}{\sqrt{t_1-t_0}} \leq \frac{y_2}{\sqrt{t_2-t_1}} \leq \cdots 
  \leq \frac{y_k}{\sqrt{t_k-t_{k-1}}} \leq 1.
  \]
  Note in particular that two nonempty subsets $\{t_0,\ldots,t_k\}$
  and $\{s_0,\ldots,s_k\}$ of $[n+1]$ label isometric $k$-dimensional
  faces if $t_i - t_{i-1} = s_i - s_{i-1}$ for all $i \in
  \{1,2,\ldots,k\}$.
\end{rem}

\begin{table}
  \begin{tabular}{c|c}
    element of $\bool(3)$ & face of the $2$-orthoscheme \\
    \hline
    \hline 
    $\{1,2,3\}$ & $0 \leq x_1 \leq x_2 \leq 1$ \\
    $\{1,2\}$ & $0 \leq x_1 \leq x_2 = 1$ \\
    $\{2,3\}$ & $0 = x_1 \leq x_2 < 1$ \\
    $\{1,3\}$ & $0 \leq x_1 = x_2 \leq 1$ \\
    $\{1\}$ & $0 \leq x_1 = x_2 = 1$ \\
    $\{2\}$ & $0 = x_1 \leq x_2 = 1$ \\
    $\{3\}$ & $0 = x_1 = x_2 \leq 1$ 
  \end{tabular}
  \vspace*{1em}
  \caption{A bijection between subsets and (closed) faces.\label{tbl:set-face}}
\end{table}

\section{Complexes}\label{sec:complexes}

In this section, we review the definition and basic properties of
interval complexes and the dual braid complex. For the rest of the
article $G$ is a group and $X$ is a fixed conjugacy-closed generating
set for $G$.

\begin{defn}[Intervals and Order Complexes]\label{def:intervals}
  Suppose $g\in G$ can be written as a product of elements in $X$
  (i.e. $g$ belongs to the \emph{monoid} generated by $X$) and let
  $\ell(g)$ denote the length of a minimal factorization of $g$ into
  elements of $X$. This induces a partial order on $G$ by declaring
  that $h \leq g$ if $h\cdot h' = g$ and $\ell(h) + \ell(h') =
  \ell(g)$; in other words, $h\leq g$ if there is a minimal-length
  factorization of $g$ into elements of $X$ which has a minimal-length
  factorization of $h$ as a left prefix. The \emph{interval poset}
  $P_g=[1,g]$ is the poset of all elements between the identity and
  $g$ in this ordering.  Observe that the order diagram or Hasse
  diagram of $P_g$ is the directed graph obtained by taking the union
  of all geodesics from $1$ to $g$ in the right Cayley graph of $G$
  with respect to $X$.  Since every geodesic has the same length, this
  makes the interval poset $P_g$ into a bounded graded poset with
  height $\ell(g)$.  The order complex of $P_g$ is the ordered
  simplicial complex $O_g = \Delta(P_g)$.
\end{defn}

The primary application of these general constructions is when $G=W$
is a Coxeter group, $X=T$ is the set of all reflections, and $g=\cb$ is a
Coxeter element of $W$.  Of particular interest is the ``type A'' case
where $G = \sym_d$, $X=T$ is the set of all transpositions, and
$g=\delta$ is the $d$-cycle $(1\ 2\ \cdots\ d)$.  Finally, the case
where $G$ is $\Z$ is a fundamental building block of the general
theory.

\begin{example}
  Let $G = \Z$ with generating set $X = \{1\}$. For each positive
  integer $n$, $\ell(n) = n$, and the induced partial order is the
  usual one for $\Z$.  The interval from $0$ to $n$ in this poset
  consists of a single chain with $n+1$ elements.
\end{example}

\begin{example}\label{ex:absolute-order-interval}
  Let $G$ be the symmetric group $\sym_d$ and let $X = T$ be the set
  of all transpositions.  Then for all $g\in \sym_d$, the length
  $\ell(g)$ is called the \emph{absolute reflection length} and the
  induced partial order is the \emph{absolute order} on the symmetric
  group. If we define $\delta$ to be the $d$-cycle $(1\ 2\ \cdots\ d)
  \in \sym_d$, then the interval $[1,\delta]$ is the \emph{lattice of
  noncrossing permutations} \cite{mccammond06},
  \cite[Section~4]{gcp2}. For example, if $\delta = (1\ 2\ 3) \in
  \sym_3$ and $T = \{a,b,c\}$ where $a = (1\ 2)$, $b = (2\ 3)$ and $c
  = (1\ 3)$, then the interval $[1,\delta]$ consists of five elements
  (see Figure~\ref{fig:absolute-order-interval}). This structure
  produces the \emph{dual presentation} for $\braid_3$, defined by
  $\langle a,b,c,\delta \mid ab = bc = ca = \delta\rangle$. More
  generally, the dual presentation for the $d$-strand braid group
  $\braid_d$ has as its generating set the nontrivial elements of
  $[1,\delta]$, with relations consisting of all words which arise
  from closed loops in $[1,\delta]$ which are based at the identity
  \cite{brady01,bessis-03}. See also \cite{McCammond-dual-euc-art,
    McCammondSulway-euc-art}.
\end{example}

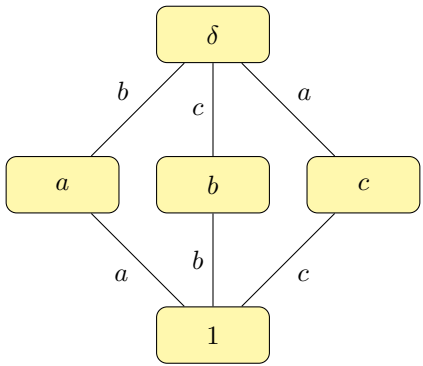
\begin{figure}
  \centering
  \begin{tikzpicture}
    \node[RowNode] (top) at (0,2) {$\delta$};
    \node[RowNode] (a1) at (-2,0) {$a$};
    \node[RowNode] (a2) at (0,0) {$b$};
    \node[RowNode] (a3) at (2,0) {$c$};
    \node[RowNode] (bottom) at (0,-2) {$1$};
    \draw (top) -- node [above left] {$b$} (a1) -- node [below left] {$a$} (bottom);
    \draw (top) -- node [left] {$c$} (a2) -- node [left] {$b$} (bottom);
    \draw (top) -- node [above right] {$a$} (a3) -- node [below right] {$c$} (bottom);
  \end{tikzpicture}
  \caption{The interval poset $P_\delta = [1,\delta]$ described in
    Example~\ref{ex:absolute-order-interval}, depicted as a subgraph
    of the Cayley graph $\text{Cay}(\sym_3,T)$.  Its order complex
    $O_g$ is a triple of right-angled Euclidean triangles with a common
    hypotenuse.}
  \label{fig:absolute-order-interval}
\end{figure}

Because the generating set $X$ is assumed to be closed under
conjugation, the interval poset $P_g = [1,g]$ also has a certain
closure property.

\begin{lem}\label{lem:interval-factors}
  If $x_1,\ldots,x_n \in G$ with $x_1\cdots x_n = g$ and $\ell(x_1) +
  \cdots + \ell(x_n) = \ell(g)$, then for any choice of integers $1
  \leq i_1 < \cdots < i_k \leq n$, we have $x_{i_1}\cdots x_{i_k} \in
       [1,g]$.
\end{lem}

\begin{proof}
  We can rewrite the factorization $g = x_1\cdots x_n$ by grouping
  terms to obtain $g = w_0 x_{i_1} w_1 x_{i_2} w_2 \cdots w_{k-1}
  x_{i_k} w_k$, where $w_j$ is defined appropriately.  Note that if
  the trivial $w_j$ are removed, this is a merged factorization of
  $g$ and 
  \[
  \ell(g) = \ell(w_0) + \ell(x_{i_1}) + \ell(w_1) + \cdots
  \ell(w_{k-1}) + \ell(x_{i_k}) + \ell(w_k). 
  \]
  If we define $z_j = x_{i_j}\cdots x_{i_k}$ for each $j$, then we can
  rearrange the first product to obtain
  \[
  g = x_{i_1} \cdots x_{i_k} (z_1^{-1} w_0 z_1) (z_2^{-1} w_1 z_2)
  \cdots (z_k^{-1} w_{k-1} z_k) w_k.
  \]
  By definition, we know that $\ell(x_{i_1} \cdots x_{i_k}) \leq
  \ell(x_{i_1}) + \cdots + \ell(x_{i_k})$, and since the generating
  set $X$ is closed under conjugation, we have $\ell(z_j^{-1} w_{j-1}
  z_j)=\ell(w_{j-1})$ for each $j$.  Combining these with the equation
  above, we obtain
  \[
  \ell(g) = \ell(x_{i_1} \cdots x_{i_k}) + \ell(w_0) + \cdots +
  \ell(w_n),
  \]
  from which it follows that $x_{i_1} \cdots x_{i_k} \leq g$.
\end{proof}

\begin{defn}[Dual braid complex]\label{def:dual-braid-complex}
  Let $g\in G$. The \emph{interval complex} $K_g$ associated to the
  interval $P_g = [1,g]$ is the single-vertex $\Delta$-complex
  obtained by identifying faces in the order complex $O_g =
  \Delta([1,g])$ as follows: the $k$-simplices labeled by chains $x_0
  < \cdots < x_k$ and $y_0 < \cdots < y_k$ are identified if and only
  if $x_{i-1}^{-1}x_i = y_{i-1}^{-1}y_i$ for all $i \in
  \{1,\ldots,k\}$. Note that this identification is well-defined by
  the ordering given to each simplex.  In particular, there is a
  well-defined surjective cellular map $q \colon O_g \onto K_g$ from
  the ordered simplicial complex $O_g$ to the one-vertex
  $\Delta$-complex $K_g$.
  In the special case when $G = \sym_d$ and $X=T$ as outlined in
  Example~\ref{ex:absolute-order-interval}, we write $\delta =
  (1\ 2\ \cdots\ d)$ and refer to the interval complex $K_\delta$ as
  the (one vertex) \emph{dual braid complex}\footnote{In some
  articles, ``dual braid complex'' refers to the universal cover of
  $K_g$, rather than $K_g$ itself.  Here, the term refers to the
  one-vertex complex $K_g$ since the universal cover is never needed.}.
\end{defn}

Brady introduced the dual braid complex in \cite{brady01} and showed
that it is a classifying space for the $d$-strand braid group. Using
the orthoscheme metric, it is conjectured that the dual braid complex
is locally $\cat(0)$ \cite{BrMc-10}. This has been shown when $d\leq
7$ but remains open for higher values of $d$
\cite{BrMc-10,haettel-kielak-schwer-16,Jeong_2023}.

\begin{rem}\label{rem:dual-braid-subcomplexes}
  If $h \leq g$, then the order complex $O_h = \Delta([1,h])$ is
  isometric to a subcomplex of $O_g = \Delta([1,g])$ and, by following
  the gluing described in Definition~\ref{def:dual-braid-complex}, we
  can see that the interval complex $K_h$ is isometric to a subcomplex
  of $K_g$. When $G = \sym_d$ and $X=T$, each permutation $\gamma \in
  \sym_d$ can be written as a product of disjoint cycles $\gamma = x_1
  \cdots x_k$ and the interval $[1,\gamma]$ is isomorphic to the
  product of intervals $[1,x_1] \times \cdots \times [1,x_k]$, so it
  follows that the interval complex $K_\gamma$ is a subcomplex of
  $K_\delta$ which is isometric to a product of smaller dual braid
  complexes.
\end{rem}

\section{Cells and Factorizations}\label{sec:factors}

In this section we develop a convenient way of describing the cells in
the order complex $O_g$ and its quotient interval complex $K_g =
q(O_g)$. To this end, we define two key posets for each $g\in G$:
$\fact(G,g,\Ib)$, the poset of linear factorizations of $g$, and
$\fact(G,g,\Sb)$, the poset of circular factorizations of $g$. As
usual, let $G$ be a group with conjugacy-closed generating set $X$ and
let $P_g = [1,g]$ be the interval poset for an element $g \in \mon(X)$
(Definition~\ref{def:intervals}).  Let $\ell\colon \mon(X) \to \Z$ be
the length function that takes each $h \in \mon(X)$ to its
\emph{length} as a word over $X$.  Note that for $h \in P_g$,
$\ell(h)$ is its rank.

\begin{defn}[Linear factorizations]\label{def:factorization}
  A \emph{linear factorization of $g \in \mon(X)$} is a row vector
  $\bx = [x_L\ x_1\ \cdots\ x_k\ x_R]$ with entries in $\mon(X)$ such that
  \begin{enumerate}
  \item $x_i$ is nontrivial when $i \in \{1,\ldots,k\}$;
  \item $\ell(x_L) + \ell(x_1) + \cdots + \ell(x_k) + \ell(x_R) = \ell(g)$;
  \item $x_L\cdot x_1\cdots x_k\cdot x_R = g$.
  \end{enumerate} 
  Note that the entries here are elements of $\mon(X)$, not
  necessarily elements in $X$.  Also note that the left and right
  entries of $\bx$ ($x_L$ and $x_R$) are labeled differently from the
  others since they are treated differently.  In particular, these
  elements are allowed to be trivial.  For each $i \in
  \{0,1,\ldots,k\}$, a linear factorization
  \[
    [x_L\  \cdots\ x_{i-1}\ x_i\ x_{i+1}\ x_{i+2} \cdots \ x_R]
  \]
  of length $k+2$ can be \emph{merged} at position $i$ to obtain the
  linear factorization
  \[
    [x_L\  \cdots\ x_{i-1}\ (x_ix_{i+1})\ x_{i+2} \cdots \ x_R]
  \]
  of length $k+1$ (where $x_0$ and $x_{k+1}$ are understood to mean
  $x_L$ and $x_R$ respectively).  Let $\fact(G,g,\Ib)$ denote the set
  of all linear factorizations of $g$, equipped with the partial order
  $\bx \leq \by$ if $\bx$ can be obtained from $\by$ by a sequence of
  merges.
\end{defn}

By Lemma~\ref{lem:interval-factors}, we see that the elements of
$\sym_d$ which appear in linear factorizations of $g$ are precisely
those which belong to the interval $[1,g]$. We also note that linear
factorizations are closely related to the ``reduced products'' in
\cite{Mc-ncht} and the ``block factorizations'' in \cite{ripoll12},
but with the slight variation that the first and last entries are
permitted to be trivial.

\begin{example}[$\fact(\sym_3,\delta,\Ib)$]\label{ex:weak-fact-2}
  If $G = \sym_3$ and $X = \{a,b,c\}$ as in
  Example~\ref{ex:absolute-order-interval}, then
  $\fact(\sym_3,\delta,\Ib)$ contains the $15$ elements shown in
  Figure~\ref{fig:weak-fact-2}.  There are $3$ in rank~$2$, $7$ in
  rank~$1$, and $5$ in rank~$0$.
\end{example}

\begin{figure}
  \centering
  \begin{tikzpicture}
    \begin{scope}
      \node[RowNode] (a1) at (-3,4) {$[1\ a\ b\ 1]$};
      \node[RowNode] (a2) at (0,4) {$[1\ c\ a\ 1]$};
      \node[RowNode] (a3) at (3,4) {$[1\ b\ c\ 1]$};
      
      \node[RowNode] (b1) at (-5.25,2) {$[a\ b\ 1]$};
      \node[RowNode] (b2) at (-3.5,2) {$[c\ a\ 1]$};
      \node[RowNode] (b3) at (-1.75,2) {$[b\ c\ 1]$};
      \node[RowNode] (b4) at (0,2) {$[1\ \delta\ 1]$};
      \node[RowNode] (b5) at (1.75,2) {$[1\ a\ b]$};
      \node[RowNode] (b6) at (3.5,2) {$[1\ c\ a]$};
      \node[RowNode] (b7) at (5.25,2) {$[1\ b\ c]$};
      
      \node[RowNode] (c1) at (-4,0) {$[\delta\ 1]$};
      \node[RowNode] (c2) at (-2,0) {$[a\ b]$};
      \node[RowNode] (c3) at (0,0) {$[c\ a]$};
      \node[RowNode] (c4) at (2,0) {$[b\ c]$};
      \node[RowNode] (c5) at (4,0) {$[1\ \delta]$};
            
      \draw (a1) -- (b1);
      \draw (a1) -- (b4);
      \draw (a1) -- (b5);
      \draw (a2) -- (b2);
      \draw (a2) -- (b4);
      \draw (a2) -- (b6);
      \draw (a3) -- (b3);
      \draw (a3) -- (b4);
      \draw (a3) -- (b7);
      
      \draw (b1) -- (c1);
      \draw (b1) -- (c2);
      \draw (b2) -- (c1);
      \draw (b2) -- (c3);
      \draw (b3) -- (c1);
      \draw (b3) -- (c4);
      \draw (b4) -- (c1);
      \draw (b4) -- (c5);
      \draw (b5) -- (c2);
      \draw (b5) -- (c5);
      \draw (b6) -- (c3);
      \draw (b6) -- (c5);
      \draw (b7) -- (c4);
      \draw (b7) -- (c5);
    \end{scope}
  \end{tikzpicture}
  \caption{The poset $\fact(\sym_3,\delta,\Ib)$ of linear
    factorizations of the $3$-cycle $\delta = (1\ 2\ 3)$ with respect
    to the generators $a= (1\ 2)$, $b = (2\ 3)$, and $c = (1\ 3)$.}
    \label{fig:weak-fact-2}
\end{figure}

\begin{example}[$\comp(\Z,n,\Ib)$]
  In the special case $G = \Z$ and $X = \{1\}$ we switch to additive
  notation and refer to the linear factorizations of $n \in \Z$ as
  \emph{linear compositions}. The set of all linear compositions is
  denoted $\comp(\Z,n,\Ib)$.  The reason for this switch to additive
  notation is to make the map $L$ from factorizations to compositions
  easier to define.  See Definition~\ref{def:fact-to-comp}.  Unlike
  the more general case, $\comp(\Z,n,\Ib)$ has a unique maximum
  element given by the row vector $[0\ 1\ \cdots\ 1\ 0]$ of length
  $n+2$. See Figure~\ref{fig:weak-comp-2} for an example when $n = 2$.
\end{example}

\begin{figure}
  \centering
  \begin{tikzpicture}
    \begin{scope}
      \node[RowNode] (top) at (0,4) {$[0\ 1\ 1\ 0]$};
      
      \node[RowNode] (a1) at (-3,2) {$[1\ 1\ 0]$};
      \node[RowNode] (a2) at (0,2) {$[0\ 2\ 0]$};
      \node[RowNode] (a3) at (3,2) {$[0\ 1\ 1]$};
      
      \node[RowNode] (b1) at (-3,0) {$[2\ 0]$};
      \node[RowNode] (b2) at (0,0) {$[1\ 1]$};
      \node[RowNode] (b3) at (3,0) {$[0\ 2]$};
            
      \draw (top) -- (a1);
      \draw (top) -- (a2);
      \draw (top) -- (a3);
      
      \draw (a1) -- (b1);
      \draw (a1) -- (b2);
      \draw (a2) -- (b1);
      \draw (a2) -- (b3);
      \draw (a3) -- (b2);
      \draw (a3) -- (b3);
    \end{scope}
  \end{tikzpicture}
  \caption{The poset $\comp(\Z,2,\Ib)$ of linear compositions of $2$.}
  \label{fig:weak-comp-2}
\end{figure}

\begin{prop}\label{prop:fact-chain-poset}
  Let $g\in G$. Then $\fact(G,g,\Ib)$ is isomorphic to the poset of
  nonempty chains for $[1,g]$, ordered by inclusion.
\end{prop}

\begin{proof}
  Let $f$ be the function which sends the linear factorization
  $[x_L\ x_1\ \cdots\ x_k\ x_R]$ to the chain
  \[x_L \leq x_Lx_1 \leq x_Lx_1x_2 \leq \cdots \leq x_Lx_1\cdots x_k,\]
  noting that $x_0x_1 \cdots x_i \in [1,g]$ for each $i$.  Then $f$
  has an inverse which takes the chain $y_0 < y_1 < \cdots < y_k$ to
  the linear factorization $[y_0\ \ y_0^{-1}y_1\ \cdots\ y_{k-1}^{-1}
    y_k\ \ y_k^{-1}g]$.  Moreover, merging factorizations in
  $\fact(G,g,\Ib)$ corresponds exactly to taking subchains in $[1,g]$,
  so $f$ is an order-preserving bijection with order-preserving
  inverse and thus the two posets are isomorphic.
\end{proof}

\begin{cor}\label{cor:comp-bool}
  $\comp(\Z,n,\Ib)$ is isomorphic to $\bool^*(n+1)$. 
\end{cor}

By identifying the ends of a linear factorization, we obtain a new object which
we call a circular factorization.

\begin{defn}[Circular factorizations]\label{def:circular-factorization}
  Let $g\in G$. Define an equivalence relation on $\fact(G,g,\Ib)$ by
  declaring $\bx = [x_L\ x_1\ \cdots\ x_k\ x_R]$ and $\by =
  [y_L\ y_1\ \cdots\ y_k\ y_R]$ to be equivalent if and only if $x_i =
  y_i$ for all $i \in \{1,\ldots,k\}$. We refer to the equivalence
  classes as \emph{circular factorizations of $g$}, each of which is
  represented by the unique element with $1$ as its final entry. More
  concretely, the equivalence class of $\bx$ is denoted
  $\overline{\bx} = [gx_Rg^{-1}x_L\ |\ x_1\ \cdots\ x_k\ |\ 1]$, where
  the vertical bars are indicators used to distinguish the first and
  last terms from the rest of the factorization.  Let $\fact(G,g,\Sb)$
  denote the set of all circular factorizations under the partial
  order $\overline{\bx} \leq \overline{\by}$ if $\bx' \leq \by'$ for
  some $\bx' \in \overline{\bx}$ and $\by' \in \overline{\by}$. Define
  the order-preserving surjection $q\colon \fact(G,g,\Ib) \to
  \fact(G,g,\Sb)$ by sending each linear factorization to the
  equivalence class which contains it, i.e. $q(\bx) = \overline{\bx}$.
\end{defn}

\begin{example}[$\fact(\sym_3,\delta,\Sb)$]\label{ex:circ-fact-2}
  When $G = \sym_3$, $X = T = \{a,b,c\}$, and $\delta = (1\ 2\ 3)$ as
  in Example~\ref{ex:absolute-order-interval}, the poset
  $\fact(\sym_3,\delta,\Sb)$ has $8$ elements. See
  Figure~\ref{fig:circ-fact-2}.
\end{example}

\begin{figure}
  \centering
  \begin{tikzpicture}
    \begin{scope}
      \node[RowNode] (a1) at (-3,4) {$[1\ | \ a\ b\ |\ 1]$};
      \node[RowNode] (a2) at (0,4) {$[1\ | \ c\ a\ |\ 1 ]$};
      \node[RowNode] (a3) at (3,4) {$[1\ | \ b\ c\ |\ 1 ]$};
      
      \node[RowNode] (b1) at (-4.5,2) {$[1\ |\ \delta\ |\ 1]$};
      \node[RowNode] (b2) at (-1.5,2) {$[c\ |\ a\ |\ 1]$};
      \node[RowNode] (b3) at (1.5,2) {$[a\ |\ b\ |\ 1]$};
      \node[RowNode] (b4) at (4.5,2) {$[b\ |\ c\ |\ 1]$};
      
      \node[RowNode] (bottom) at (0,0) {$[\delta\ ||\ 1]$};
      
      \draw (a1) -- (b1);
      \draw (a1) -- (b2);
      \draw (a1) -- (b3);
      \draw (a2) -- (b1);
      \draw (a2) -- (b2);
      \draw (a2) -- (b4);
      \draw (a3) -- (b1);
      \draw (a3) -- (b3);
      \draw (a3) -- (b4);
      
      \draw (b1) -- (bottom);
      \draw (b2) -- (bottom);
      \draw (b3) -- (bottom);
      \draw (b4) -- (bottom);
    \end{scope}
  \end{tikzpicture}
  \caption{The poset $\fact(\sym_3,\delta,\Sb)$ of circular
    factorizations of the $3$-cycle $\delta = (1\ 2\ 3)$ with respect
    to the generators $a= (1\ 2)$, $b = (2\ 3)$, and $c = (1\ 3)$.}
    \label{fig:circ-fact-2}
\end{figure}

\begin{example}[$\comp(\Z,n,\Sb)$]\label{ex:circ-comp-2}
  When $G = \Z$ and $X = \{1\}$, we denote the set $\fact(\Z,g,\Sb)$
  by $\comp(\Z,n,\Sb)$ and refer to its elements as \emph{circular
  compositions of $n$}. When $n = 2$, for example, the poset
  $\comp(\Z,2,\Sb)$ of circular compositions of $2$ has four elements
  and is isomorphic to $\bool(2)$ --- see
  Figure~\ref{fig:circ-comp-2}. More generally, it is important to
  note that while there is a rank-preserving bijection between
  $\comp(\Z,n,\Sb)$ and $\bool(n)$, the two are not isomorphic when $n
  > 2$ since there are relations in the former which do not appear in
  the latter. For example, each rank-one element of $\bool(3)$ lies
  beneath exactly two rank-two elements, but the circular composition
  $[1\ |\ 2\ |\ 0]$ lies beneath $[1\ |\ 1\ 1\ |\ 0]$,
  $[0\ |\ 1\ 2\ |\ 0]$ and $[0\ |\ 2\ 1\ |\ 0]$ in $\comp(\Z,3,\Sb)$.
\end{example}

\begin{figure}
  \centering
  \begin{tikzpicture}
    \begin{scope}
      \node[RowNode] (top) at (0,2) {$[0\ |\ 1\ 1\ |\ 0]$};
      
      \node[RowNode] (a1) at (-1.5,0) {$[1\ |\ 1\ |\ 0]$};
      \node[RowNode] (a2) at (1.5,0) {$[0\ |\ 2\ |\ 0]$};
      
      \node[RowNode] (bottom) at (0,-2) {$[2\ | |\ 0]$};
      
      \draw (top) -- (a1);
      \draw (top) -- (a2);
      
      \draw (a1) -- (bottom);
      \draw (a2) -- (bottom);
    \end{scope}
  \end{tikzpicture}
  \caption{The poset $\comp(\Z,2,\Sb)$ of circular compositions of $2$.}
  \label{fig:circ-comp-2}
\end{figure}

We close this section by providing combinatorial connections between
the posets introduced in this section. See
Figure~\ref{fig:compositions-factorizations} for the relevant
commutative diagram.

\begin{defn}[Factorizations to compositions]\label{def:fact-to-comp}
  Let $g\in G$ with $n = \ell(g)$ and define the function $L \colon
  \fact(G,g,\Ib) \to \comp(\Z,n,\Ib)$ by sending the factorization
  $\bx =$ $[x_L\ x_1\ \cdots\ x_k\ x_R]$ to the composition $L(\bx) =$
  $[\ell(x_L)\ \ell(x_1)\ \cdots\ \ell(x_k)\ \ell(x_R)]$.  Similarly,
  define $\overline{L} \colon \fact(G,g,\Sb) \to \comp(\Z,n,\Sb)$ by
  sending $\overline{\bx} =$ $[x_L\ |\ x_1\ \cdots\ x_k\ |\ 1]$ to
  $L(\overline{\bx}) =$
  $[\ell(x_L)\ |\ \ell(x_1)\ \cdots\ \ell(x_k)\ |\ 0]$.  Note that if
  $\bx'$ is obtained from $\bx$ by performing a merge in position $i$,
  then we know $\ell(x_ix_{i+1}) = \ell(x_i) + \ell(x_{i+1})$ by
  Lemma~\ref{lem:interval-factors}, and this means that $L(\bx')$ is
  obtained from $L(\bx)$ by merging at position $i$. By similar
  reasoning, we know that if $\overline{\bx'} \leq \overline{\bx}$,
  then $\overline{L}(\overline{\bx'}) \leq
  \overline{L}(\overline{x})$. Thus $L$ and $\overline{L}$ are both
  order-preserving functions.
\end{defn}

\begin{lem}\label{lem:lower-fact-to-lower-comp}
    Let $\bx \in \fact(G,g,\Ib)$. Then $L$ restricts to an isomorphism
    from the lower set $\low(\bx)$ to the lower set $\low(L(\bx))$.
\end{lem}

\begin{proof}
  As described in Definition~\ref{def:fact-to-comp}, each element of
  the lower set $\low(\bx)$ is obtained from $\bx$ via a sequence of
  merges, and the same merges can be performed on $L(\bx)$ to obtain
  an element of $\low(L(\bx))$. Two different sequences of merges
  produce the same element of $\low(\bx)$ if and only if the analogous
  sequence produces the same element of $\low(L(\bx))$, so $L$
  restricts to a bijection and thus an isomorphism.
\end{proof}

Note that Lemma~\ref{lem:lower-fact-to-lower-comp} does not hold for
$\overline{L}$, as it is possible to have elements $\overline{\bx} \in
\fact(G,g,\Sb)$ such that the restriction of $\overline{L}$ to
$\low(\overline{\bx})$ is an order-preserving surjection that is not
injective. Compare Figures~\ref{fig:circ-fact-2}
and~\ref{fig:circ-comp-2} for an example.

\begin{prop}\label{prop:fact-to-comp}
  The functions $L$ and $\overline{L}$ are surjective order-preserving
  maps and $\overline{L}q = qL$.
\end{prop}

\begin{proof}
  By Definition~\ref{def:fact-to-comp}, we know the two maps are
  order-preserving.  Also, if $\bx$ is a maximal element of
  $\fact(G,g,\Ib)$, then $L(\bx) = [0\ 1\ \cdots\ 1\ 0]$, the maximum
  element of $\comp(\Z,n,\Ib)$.  By
  Lemma~\ref{lem:lower-fact-to-lower-comp}, we know that $L$ restricts
  to an isomorphism from $\low(\bx)$ to $\low(L(\bx)) =
  \comp(\Z,n,\Ib)$, so the map $L$ is surjective.  The fact that
  $\overline{L}q = qL$ follows directly from
  Definition~\ref{def:circular-factorization} and
  Example~\ref{ex:circ-comp-2}. Since $q$ and $L$ are surjective, $qL$
  is surjective and $\overline{L}q = qL$ is surjective.  Thus
  $\overline{L}$ must also be surjective.
\end{proof}

\begin{figure}
    \centering
    \begin{tikzpicture}
        \node (a) at (-2,1) {$\fact(G,g,\Ib)$};
        \node (b) at (2,1) {$\comp(\Z,n,\Ib)$};
        \node (c) at (-2,-1) {$\fact(G,g,\Sb)$};
        \node (d) at (2,-1) {$\comp(\Z,n,\Sb)$};
        \draw[->] (a) -- node [above] {$L$} (b);
        \draw[->] (a) -- node [left] {$q$} (c);
        \draw[->] (b) -- node [right] {$q$} (d);
        \draw[->] (c) -- node [above] {$\overline{L}$} (d);
    \end{tikzpicture}
    \caption{Our four key posets fit into a commutative diagram.}
    \label{fig:compositions-factorizations}
\end{figure}

\begin{prop}\label{prop:fact-intervals}
    The poset $\fact(G,g,\Ib)$ is \emph{simplicial}, i.e. each of its
    intervals is isomorphic to a Boolean lattice.
\end{prop}

\begin{proof}
  For each $\bx,\by \in \fact(G,g,\Ib)$ with $\bx \leq \by$, choose a
  maximal element $\bz$ with $\bx \leq \by \leq \bz$. Then
  the interval $[\bx, \by]$ is contained within the lower set 
  $\low(\bz)$, which we know by
  Lemma~\ref{lem:lower-fact-to-lower-comp} is isomorphic to
  $\comp(\Z,n,\Ib)$, which is itself a truncated Boolean lattice 
  (Corollary~\ref{cor:comp-bool}). Since every interval of a truncated 
  Boolean lattice is isomorphic to a smaller Boolean lattice, the proof 
  is complete.
\end{proof}

\section{Points and Weighted Factorizations}\label{sec:wt-factors}

In this section, we introduce a way of labeling individual points in
the order complex $O_g$ and the interval complex $K_g$ by adding in
weights. For each $g \in G$, we define $\wfact(G,g,\Ib)$, the space of
weighted linear factorizations of $g$, and $\wfact(G,g,\Sb)$, the
space of weighted circular factorizations of $g$. The points in each
space are defined as weighted versions of poset elements from the
previous section, or as decorated multisets in the interval $\Ib$ or
the circle $\Sb$.

\begin{defn}[$G$-multisets]\label{def:g-multisets}
  Let $S$ be a set and let $G$ be a group. We define a
  \emph{$G$-multiset} on $S$ to be a function $\bx \colon S \to G$
  such that $\bx(s)$ is the identity in $G$ for all but finitely many
  $s\in S$. We denote the set of all $G$-multisets on $S$ by
  $\mult(G,S)$.
\end{defn}

We are interested in eight cases which arise from two choices: $S$ is
either the unit interval $\Ib$ or the circle $\Sb$, and $G$ is a
general group, a Coxeter group $W$, the symmetric group $\sym_d$, or
the integers $\Z$. See Tables~\ref{tbl:linear} and~\ref{tbl:circular}
in the Introduction. The $G$-multisets arising from these cases are
slightly unusual in that they are functions from uncountable sets to
discrete groups which produce nontrivial output for only finitely many
inputs, but the interpretation is natural: one should picture the set
$S$ with a finite number of special points labeled by non-trivial
elements of $G$.

\begin{defn}[$\wfact(G,g,\Ib)$]
  Let $g\in G$. For each $G$-multiset $\bu \colon \Ib \to G$, let $0 =
  s_L < s_1 < \cdots < s_k < s_R = 1$ be such that $\bu$ is nontrivial
  on the set $\{s_1,\ldots,s_k\}$ and trivial on its complement in
  $(0,1)$.  The $G$-multiset $\bu$ may or may not be trivial on $0$
  and $1$.  For each $i$, let $x_i = \bu(s_i)$ and define $P(\bu) =
  [x_L\ x_1\ \cdots\ x_k\ x_R]$; we say that $\bu$ is a \emph{weighted
  linear factorization} of $g$ if $P(\bu)$ is a linear factorization
  of $g$. Note that $P(\bu)$ necessarily has length at least
  $2$. Denote the set of all weighted linear factorizations of $g$ by
  $\wfact(G,g,\Ib)$ and observe that $P$ is a surjective function $P
  \colon \wfact(G,g,\Ib) \onto \fact(G,g,\Ib)$. We will often use the
  convenient shorthand $\bu = 0^{x_L}s_1^{x_1}\cdots s_k^{x_k}1^{x_R}$
  to denote elements of $\wfact(G,g,\Ib)$.
\end{defn}

\begin{figure}
  \centering
  \begin{tikzpicture}

    \tikzstyle{multbackground} = [
        shape = rectangle, 
        draw = black,
        fill = yellow!40, 
        rounded corners, 
        minimum height = 1.5cm, 
        minimum width = 4cm
    ]
    
    \newcommand{\makemultiset}[2]{
            \node[multbackground] at (0,0) {};
            \node[dot, label = {\Large #1}] (x0) at (-1.5cm,0) {};
            \node[dot, label = {\Large #2}] (x1) at (1.5cm,0) {};
            \draw (x0) -- (x1);
    }
    
    \begin{scope}[x={(4cm,1cm)},y={(-.5cm,2cm)},z={(0cm,4cm)}]
        \pgfmathsetmacro{\z}{4};
        \pgfmathsetmacro{\t}{sqrt(3)};
        \coordinate (d0) at (0,0,0);
        \coordinate (a) at (1,\t,\z/2);
        \coordinate (b) at (-2,0,\z/2);
        \coordinate (c) at (1,-\t,\z/2);
        \coordinate (d1) at (0,0,\z);
    \end{scope}

    \begin{scope}[GreenPoly, line width = 0.4mm, opacity = 0.8]
        \filldraw (d0) -- (a) -- (d1) -- cycle;
        \filldraw (d0) -- (b) -- (d1) -- cycle;
        \filldraw (d0) -- (c) -- (d1) -- cycle;
    \end{scope}

    \node[dot] at (d0) {};
    \node[dot] at (a) {};
    \node[dot] at (b) {};
    \node[dot] at (c) {};
    \node[dot] at (d1) {};
    
    \node[dot] (e1) at (barycentric cs:d0=0.7,c=0.3,d1=0) {};
    \node[dot] (e2) at (barycentric cs:d0=0,a=0.7,d1=0.3) {};
    \node[dot] (e3) at (barycentric cs:d0=0.2,d1=0.8) {};
    \node[dot] (f1) at (barycentric cs:d0=0.4,b=0.4,d1=0.2) {};

    \coordinate (d0-label) at (0,-3);
    \coordinate (a-label) at (8,11);
    \coordinate (b-label) at (-11,6);
    \coordinate (c-label) at (9,5);

    \coordinate (d1-label) at (0,19);
    \coordinate (e1-label) at (6,0);
    \coordinate (e2-label) at (5,16);
    \coordinate (e3-label) at (-6,15);

    \coordinate (f1-label) at (-7,1);

    \foreach \x in {d0,a,b,c,d1,e1,e2,e3,f1} {
        \draw[->,shorten >=0.2cm] (\x-label) -- (\x);
    }

    \begin{scope}[shift = {(d0-label)}]
        \makemultiset{$\delta$}{$1$}
    \end{scope}

    \begin{scope}[shift = {(a-label)}]
        \makemultiset{$c$}{$a$}
    \end{scope}

    \begin{scope}[shift = {(b-label)}]
        \makemultiset{$a$}{$b$}
    \end{scope}

    \begin{scope}[shift = {(c-label)}]
        \makemultiset{$b$}{$c$}
    \end{scope}

    \begin{scope}[shift = {(d1-label)}]
        \makemultiset{$1$}{$\delta$}
    \end{scope}

    \begin{scope}[shift = {(e1-label)}]
        \makemultiset{$b$}{$1$}
        \node[dot,label={\Large $c$}] at ($(x0)!0.3!(x1)$) {}; 
    \end{scope}

    \begin{scope}[shift = {(e2-label)}]
        \makemultiset{$1$}{$a$}
        \node[dot,label={\Large $c$}] at ($(x0)!0.3!(x1)$) {}; 
    \end{scope}

    \begin{scope}[shift = {(e3-label)}]
        \makemultiset{$1$}{$1$}
        \node[dot,label={\Large $\delta$}] at ($(x0)!0.8!(x1)$) {}; 
    \end{scope}

    \begin{scope}[shift = {(f1-label)}]
        \makemultiset{$1$}{$1$}
        \node[dot,label={\Large $a$}] at ($(x0)!0.333!(x1)$) {}; 
        \node[dot,label={\Large $b$}] at ($(x0)!0.5!(x1)$) {}; 
    \end{scope}

\end{tikzpicture}
  \caption{The order complex $O_\delta = \Delta([1,\delta])$ has
    points labeled by elements of $\wfact(\sym_3,\delta,\Ib)$ and
    cells labeled by elements of $\fact(\sym_3,\delta,\Ib)$.}
  \label{fig:weighted-fact-2-interval}
\end{figure}
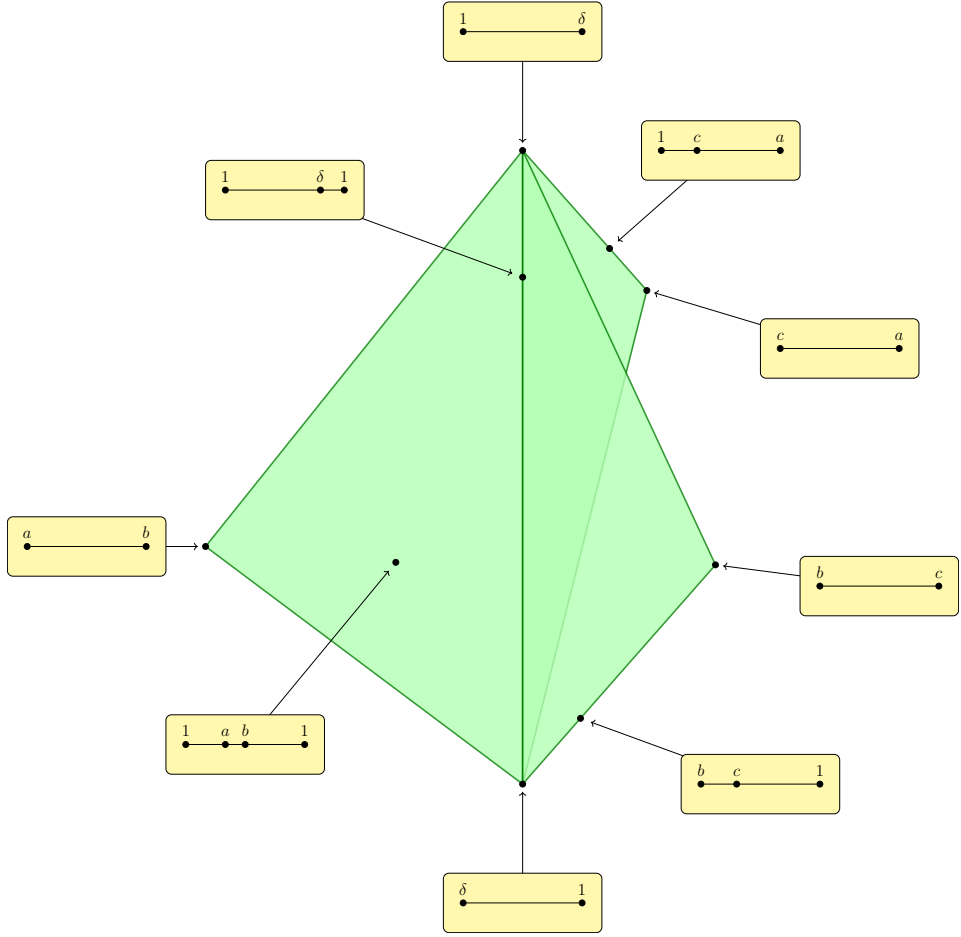

\begin{example}[$\wfact(\sym_3,\delta,\Ib)$]\label{ex:wfact-g-interval}
  If $G = \sym_3$, $T = \{a,b,c\}$ and $\delta = (1\ 2\ 3)$ as in
  Example~\ref{ex:absolute-order-interval}, then
  $\wfact(\sym_3,\delta,\Ib)$ is a $2$-dimensional simplicial complex
  which consists of three triangles, all sharing a common edge.  The
  $3$ triangles, $7$ edges, and $5$ vertices of $O_\delta =
  \wfact(\sym_3,\delta,\Ib)$ shown in
  Figure~\ref{fig:weighted-fact-2-interval} are labeled by the $15$
  elements of $\fact(\sym_3,\delta,\Ib)$ shown in
  Figure~\ref{fig:weak-fact-2}.
\end{example}

\begin{example}[$\wcomp(\Z,n,\Ib)$]\label{ex:wcomp-n-interval}
  When $G = \Z$ and $X = \{1\}$, we denote the set $\wfact(\Z,n,\Ib)$
  by $\wcomp(\Z,n,\Ib)$ and refer to its elements as \emph{weighted
  linear compositions of $n$}. If we consider the action of $\sym_n$
  on the $n$-cube $\Ib^n$ by permuting coordinates, then each element
  $\bs = 0^{a_L}s_1^{a_1}\cdots s_k^{a_k} 1^{a_R}$ in
  $\wcomp(\Z,n,\Ib)$ can be viewed as a point on the quotient space
  $\Ib^n/\sym_n$, which is isometric to a standard $n$-dimensional
  orthoscheme. For each $\ba \in \comp(\Z,n,\Ib)$, the set of weighted
  linear compositions $\bs$ of $n$ with $P(\bs) = \ba$ forms an open
  face of the standard orthoscheme which we refer to as a
  (non-standard) \emph{orthoscheme of shape $\ba$}.  This recovers
  what we found in Corollary~\ref{cor:comp-bool}: $\comp(\Z,n,\Ib)$ is
  the face poset for the standard $n$-dimensional orthoscheme.
  Moreover, Remark~\ref{rem:orthoscheme-faces} tells us that elements
  $\ba = [a_L\ a_1\ \cdots\ a_k\ a_R]$ and $\bb =
  [b_L\ b_1\ \cdots\ b_k\ b_R]$ in $\comp(\Z,n,\Ib)$ label isometric
  faces of $\wcomp(\Z,n,\Ib)$ if $a_i = b_i$ for all $i \in
  \{1,\ldots,k\}$; see Figure~\ref{fig:mult-3}.
\end{example}

\begin{figure}
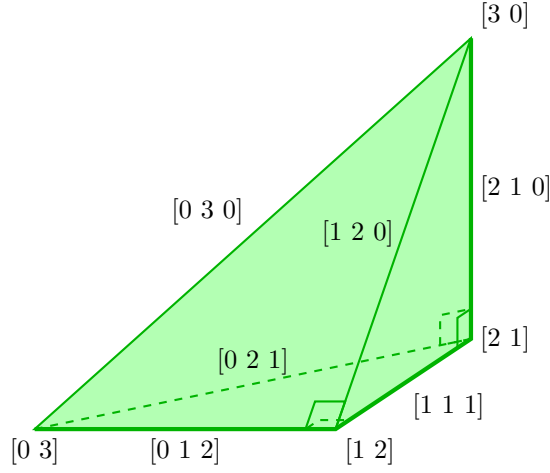

  \centering \includestandalone{fig-orthoscheme}
  \caption{The space $\wcomp(\Z,3,\Ib)$ is a $3$-dimensional
    orthoscheme.  Here we have labeled its vertices and open edges
    by the corresponding elements of $\comp(\Z,3,\Ib)$.} \label{fig:mult-3}
\end{figure}

We can use the orthoscheme metric on $\wcomp(\Z,n,\Ib)$ and an
extension of the map $L$ to provide a piecewise-Euclidean metric for
$\wfact(G,g,\Ib)$.

\begin{defn}[Pullback metric]\label{def:pullback-metric}
  Let $g\in G$ with $\ell(g) = n$. Abusing notation, we define the
  function $L\colon \wfact(G,g,\Ib) \to \wcomp(\Z,n,\Ib)$ by $L(\bu) =
  \ell \circ \bu$ and observe that $PL = LP$, where the second
  occurrence of ``$L$'' denotes the function from $\fact(G,g,\Ib)$ to
  $\comp(\Z,n,\Ib)$ given in Definition~\ref{def:fact-to-comp}.  For
  each $\bx \in \fact(G,g,\Ib)$, the set $P^{-1}(\bx) = \{\bu \in
  \wfact(G,g,\Ib) \mid P(\bu) = \bx\}$ is sent bijectively via $L$ to
  the set $P^{-1}(L(\bx)) = \{\bs \in \wcomp(\Z,n,\Ib) \mid P(\bs) =
  L(\bx)\}$, so we can pull back the metric and identify $P^{-1}(\bx)$
  with an open orthoscheme of shape $L(\bx)$.  By
  Lemma~\ref{lem:lower-fact-to-lower-comp}, we know that the closure
  of this open orthoscheme is indeed the closed orthoscheme we would
  expect, so this endows $\wfact(G,g,\Ib)$ with the structure of a
  piecewise-Euclidean $\Delta$-complex with $\fact(G,g,\Ib)$ as its
  face poset.
\end{defn}

The pullback metric defined in Definition~\ref{def:pullback-metric}
turns the simplicial complex shown in
Figure~\ref{fig:weighted-fact-2-interval} into $3$ isosceles right
triangles with a common hypotenuse. Note that the function $L$ from
$\wfact(\sym_3,\delta,\Ib)$ to $\wcomp(\Z,2,\Ib)$ is a branched
covering map with branch points along the shared hypotenuse.

\begin{prop}\label{prop:wfact-order-complex}
  The order complex $O_g = \Delta([1,g])$ is isometric to the space
  $\wfact(G,g,\Ib)$ of weighted linear factorizations of $g$.
\end{prop}

\begin{proof}
  This follows from
  Proposition~\ref{prop:fact-chain-poset} and
  Example~\ref{ex:wfact-g-interval}. More explicitly, the differences
  between the consecutive real numbers $0 = s_L < s_1 < \cdots < s_k <
  s_R = 1$ are the barycentric coordinates of a point in the corresponding
  open simplex.
\end{proof}

Identifying the endpoints of $\Ib$ to produce the circle $\Sb$ yields
a ``circular'' quotient of $\wfact(G,g,\Ib)$.

\begin{defn}[$\wfact(G,g,\Sb)$]\label{def:weighted-fact-circle}
  The equivalence relation given in
  Definition~\ref{def:circular-factorization} transforms
  $\fact(G,g,\Ib)$ into $\fact(G,g,\Sb)$, and this determines a
  quotient of the cell complex $\wfact(G,g,\Ib)$ by isometrically
  identifying faces.  As described in
  Definition~\ref{def:orthoscheme}, each simplex in the order complex
  comes with an ordering of its vertices which determines the metric,
  and this information determines the gluing orientation.  We refer to
  this quotient as the space of \emph{weighted circular factorizations
  of $g$}, denoted $\wfact(G,g,\Sb)$. Each point $\overline{\bu}$ in
  this space can be viewed as a $G$-multiset $\Sb \to G$ and uniquely
  represented as $\overline{\bu} = 0^{x_L}s_1^{x_1}\cdots
  s_k^{x_k}1^1$, where $q(\overline{\bu})$ is the circular
  factorization $[x_L\ |\ x_1\ \cdots\ x_k\ |\ 1]$.  It follows that
  $\fact(G,g,\Sb)$ is the face poset for $\wfact(G,g,\Sb)$.
\end{defn}

\begin{figure}
  \centering
  \begin{tikzpicture}

    \tikzstyle{multbackground} = [
        shape = rectangle, 
        draw = black,
        fill = yellow!40, 
        rounded corners, 
        minimum height = 4cm, 
        minimum width = 4cm
    ]
    
    \newcommand{\makemultiset}[1]{
            \node[multbackground] at (0,0) {};
            \draw (0,0) circle (1.25cm);
            \node[dot, label = {[right, shift={(0.15cm,0)}] \Large #1}] (x0) at (0:1.25cm) {};
    }
    
    \begin{scope}[x={(4cm,1cm)},y={(-.5cm,2cm)},z={(0cm,4cm)}]
        \pgfmathsetmacro{\z}{4};
        \pgfmathsetmacro{\t}{sqrt(3)};
        \coordinate (d0) at (0,0,0);
        \coordinate (a) at (1,\t,\z/2);
        \coordinate (b) at (-2,0,\z/2);
        \coordinate (c) at (1,-\t,\z/2);
        \coordinate (d1) at (0,0,\z);
    \end{scope}

    \begin{scope}[GreenPoly, line width = 0.4mm, opacity = 0.8]
        \filldraw (d0) -- (a) -- (d1) -- cycle;
        \filldraw (d0) -- (b) -- (d1) -- cycle;
        \filldraw (d0) -- (c) -- (d1) -- cycle;
    \end{scope}

    \node[dot] at (d0) {};
    \node[dot] at (a) {};
    \node[dot] at (b) {};
    \node[dot] at (c) {};
    \node[dot] at (d1) {};
    
    \node[dot] (e1) at (barycentric cs:d0=0.7,c=0.3,d1=0) {};
    \node[dot] (e2) at (barycentric cs:d0=0,a=0.7,d1=0.3) {};
    \node[dot] (e3) at (barycentric cs:d0=0.2,d1=0.8) {};
    \node[dot] (f1) at (barycentric cs:d0=0.4,b=0.4,d1=0.2) {};

    \coordinate (d0-label) at (0,-3);
    \coordinate (a-label) at (8,11);
    \coordinate (b-label) at (-11,6);
    \coordinate (c-label) at (9,5);

    \coordinate (d1-label) at (0,19);
    \coordinate (e1-label) at (6,0);
    \coordinate (e2-label) at (5,16);
    \coordinate (e3-label) at (-6,15);

    \coordinate (f1-label) at (-7,1);

    \foreach \x in {d0,a,b,c,d1,e1,e2,e3,f1} {
        \draw[->,shorten >=0.2cm] (\x-label) -- (\x);
    }

    \foreach \x in {d0,a,b,c,d1} {
        \begin{scope}[shift = {(\x-label)}]
            \makemultiset{$\delta$}
        \end{scope}
    }

    \begin{scope}[shift = {(e1-label)}]
        \makemultiset{$b$}
        \node[dot,label={[above left, shift={(0cm,-0.04cm)}] \Large $c$}] at (108:1.25cm) {}; 
    \end{scope}

    \begin{scope}[shift = {(e2-label)}]
        \makemultiset{$b$}
        \node[dot,label={[above left, shift={(0cm,-0.04cm)}] \Large $c$}] at (108:1.25cm) {}; 
    \end{scope}

    \begin{scope}[shift = {(e3-label)}]
        \makemultiset{$1$}
        \node[dot,label={[below, shift={(0.08cm,-0.08cm)}] \Large $\delta$}] at (288:1.25cm) {}; 
    \end{scope}

    \begin{scope}[shift = {(f1-label)}]
        \makemultiset{$1$}
        \node[dot,label={[above left, shift={(0cm,-0.04cm)}] \Large $a$}] at (120:1.25cm) {}; 
        \node[dot,label={[left, shift={(-0.1cm,0)}] \Large $b$}] at (180:1.25cm) {}; 
    \end{scope}

\end{tikzpicture}
  \caption{The interval complex $K_\delta = \wfact(\sym_3,\delta,\Sb)$
    is obtained by gluing the points of the order complex $O_\delta =
    \wfact(\sym_3,\delta,\Ib)$ as described in
    Definition~\ref{def:weighted-fact-circle}.  The order complex is
    shown with its circular labels.  From the circular labeling, one
    can see that all $5$ vertices are identified and the short edges
    are identified in pairs.}
  \label{fig:weighted-fact-2-circle}
\end{figure}
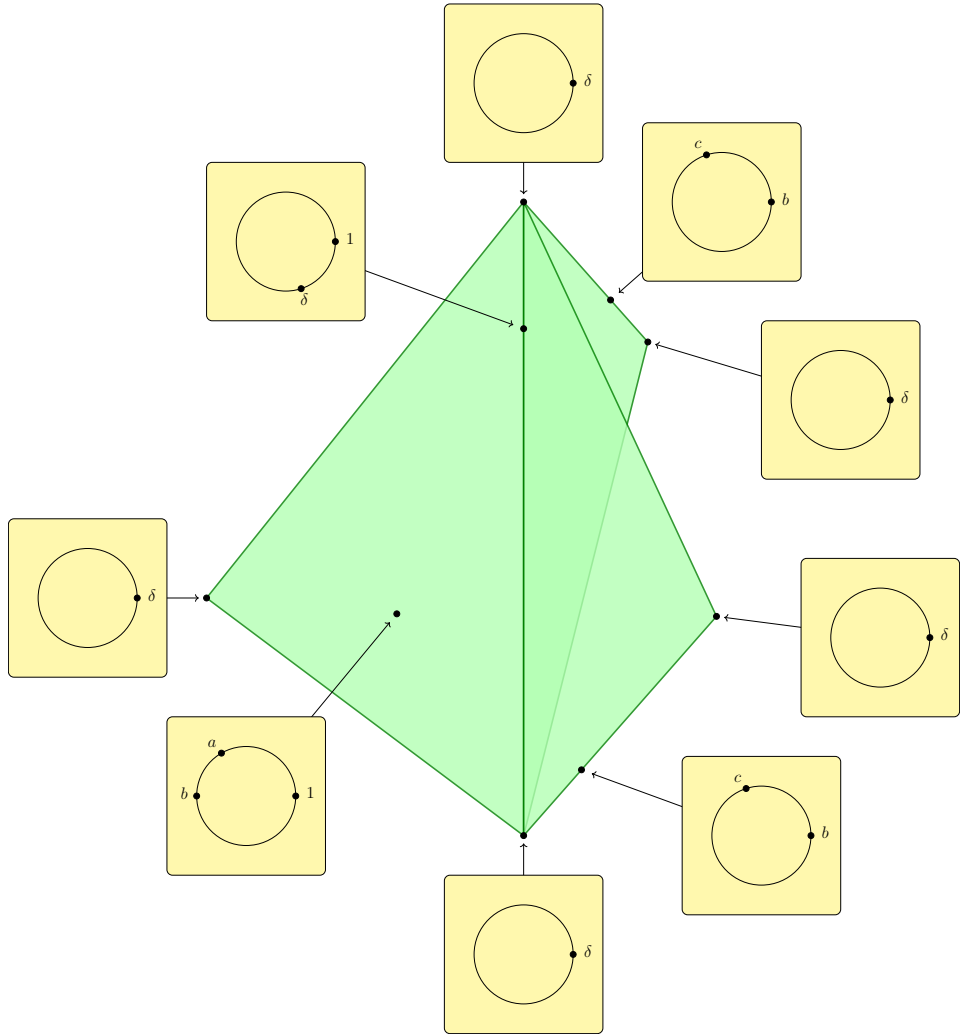

\begin{example}[$\wfact(\sym_3,\delta,\Sb)$]\label{ex:wfact-g-circle}
  If $G = \sym_3$, $T = \{a,b,c\}$ and $\delta = (1\ 2\ 3)$ as in
  Example~\ref{ex:wfact-g-interval}, then $\wfact(\sym_3,\delta,\Sb)$
  is obtained from $\wfact(\sym_3,\delta,\Ib)$ by identifying edges.
  The quotient $K_g = \wfact(\sym_3,\delta,\Sb)$ has $3$ triangles,
  $4$ edges, and $1$ vertex, and these cells are labeled by the $8$
  elements of $\fact(\sym_3,\delta,\Sb)$ shown in
  Figure~\ref{fig:circ-fact-2}. 
  For a depiction of the pointwise labels in $K_g$, see 
  Figure~\ref{fig:weighted-fact-2-circle}. 
\end{example}

\begin{example}[$\wcomp(\Z,n,\Sb)$]
    \label{ex:weighted-comp-circle}
    When $G = \Z$ and $X = \{1\}$, we denote the set
    $\wfact(\Z,n,\Sb)$ by $\wcomp(\Z,n,\Sb)$ and refer to its elements
    as \emph{weighted circular compositions of $n$}.  As discussed in
    Example~\ref{ex:wcomp-n-interval}, $\wcomp(\Z,n,\Ib)$ is isometric
    to a standard $n$-dimensional orthoscheme, so $\wcomp(\Z,n,\Sb)$
    is obtained by identifying faces of an orthoscheme according to
    the map $q \colon \comp(\Z,n,\Ib) \to \comp(\Z,n,\Sb)$ given in
    Definition~\ref{def:circular-factorization}. To give another way
    of viewing this identification, the inequalities $x_1 \leq x_2
    \leq \cdots \leq x_n \leq x_1 + 1$ define a topological subspace
    of $\mathbb{R}^n$ called a \emph{column} which is isometric to the
    product of $\mathbb{R}$ and an $(n-1)$-simplex (more specifically,
    a Coxeter simplex of type $\widetilde{A}_{n-1}$. See
    \cite[Section~8]{BrMc-10} and \cite[Section~8]{dmw20}).  The
    infinite cyclic group generated by the isometry $T:\mathbb{R}^n
    \to \mathbb{R}^n$ given by $T(x_1,x_2,\ldots,x_n) =
    (x_2,\ldots,x_n,x_1+1)$ acts freely on the column and a
    fundamental domain for the action is a standard $n$-dimensional
    orthoscheme. The quotient by this free $\Z$-action is
    $\wcomp(\Z,n,\Sb)$.
\end{example}

\begin{prop}\label{prop:wfact-interval-complex}
  The interval complex $K_g = q(O_g)$ for the interval $P_g=[1,g]$ is
  isometric to the space $\wfact(G,g,\Sb)$ of weighted circular factorizations
  of $g$.
\end{prop}

\begin{proof}
  The interval complex $K_g$ for $[1,g]$ is obtained from the order
  complex by identifying the faces labeled by chains $x_0 < \cdots <
  x_k$ and $y_0 < \cdots < y_k$ if and only if $x_{i-1}^{-1}x_i =
  y_{i-1}^{-1}y_i$ for all $i \in \{1,\ldots,k\}$.  By
  Proposition~\ref{prop:fact-chain-poset}, these faces are labeled by
  the linear factorizations
    \[[x_L\ \ x_L^{-1}x_1\ \cdots\ x_{k-1}^{-1} x_k\ \ x_k^{-1}g]\]
    and
    \[[y_L\ \ y_L^{-1}y_1\ \cdots\ y_{k-1}^{-1} y_k\ \ y_k^{-1}g],\]
    so the identification of faces in constructing the interval
    complex is identical to the equivalence relation established on
    $\fact(G,g,\Ib)$ when defining $\fact(G,g,\Sb)$. Since this
    equivalence relation dictates the identification of faces in
    $\wfact(G,g,\Ib)$ when constructing $\wfact(G,g,\Sb)$ and
    $\wfact(G,g,\Ib)$ is isometric to the ordered simplicial complex
    $O_g=\Delta([1,g])$ by Proposition~\ref{prop:wfact-order-complex},
    the proof is complete.
\end{proof}

\section{Topological Posets}\label{sec:top-posets}

In this section we define two graded topological posets: a linear
version $\Fc(G,g,\Ib)$, which contains the weighted linear
factorizations of all elements $h$ in $P_g = [1,g]$, and a circular
version $\Fc(G,g,\Sb)$, which is a quotient of the linear version.  We
also prove Theorems~\ref{mainthm:max-elements},
\ref{mainthm:dual-complex}, and~\ref{mainthm:upper-sets}.  We
begin by recalling the definition.

\begin{defn}[Topological posets]\label{def:top-poset}
 A \emph{topological poset} is a poset $\mathcal{P}$ with a Hausdorff
 topology $\tau$ such that the order relation $R := \{(p, q) \in
 \mathcal{P} \times \mathcal{P} \mid p \leq q\}$ is a closed subspace
 of $\mathcal{P} \times \mathcal{P}$.  Morphisms are continuous poset
 maps.
\end{defn}

This definition was given by Rade \v{Z}ivaljevi\'c \cite{zivaljevic98}
inspired by the work of Vasiliev \cite{vasiliev91}.  The idea of
combining posets and topology goes back to Folkman \cite{folkman66}
and Bj\"{o}rner \cite{bjorner95} has a survey of the area from the
mid-1990s.  For most natural examples, including the two given below,
establishing the technical condition in Definition~\ref{def:top-poset}
is straightforward, and it is an exercise best left to the reader.

The linear topological poset $\Fc(G,g,\Ib)$ is easy to define and we
establish its properties before moving on to the more complicated
circular version $\Fc(G,g,\Sb)$.

\begin{defn}[$\Fc(G,g,\Ib)$]\label{def:fact-top-poset}
  For each $g\in G$, define $\Fc(G,g,\Ib)$ to be the disjoint union of
  topological spaces
  \[
  \bigsqcup_{h\in [1,g]} \wfact(G,h,\Ib).
  \]
  This space comes with a partial order: for $\bu,\bv \in
  \Fc(G,g,\Ib)$, we say that $\bv$ is a \emph{subfactorization} of
  $\bu$ and write $\bv \subseteq \bu$ if $\bv(r) \leq \bu(r)$ in $G$
  for all $r\in \Ib$. Define $\rho \colon \Fc(G,g,\Ib) \to G$ by $\rho
  (\bu) = \prod_{r\in [0,1]} \bu(r)$, where the factors are arranged
  from left to right in increasing order of the real numbers $r$. Note
  that this uncountable product is well-defined since all but finitely
  many factors are trivial.  The topological poset $\Fc(G,g,\Ib)$ is a
  graded poset of height $\ell(g)$ with rank function $\ell\circ\rho$.
\end{defn}

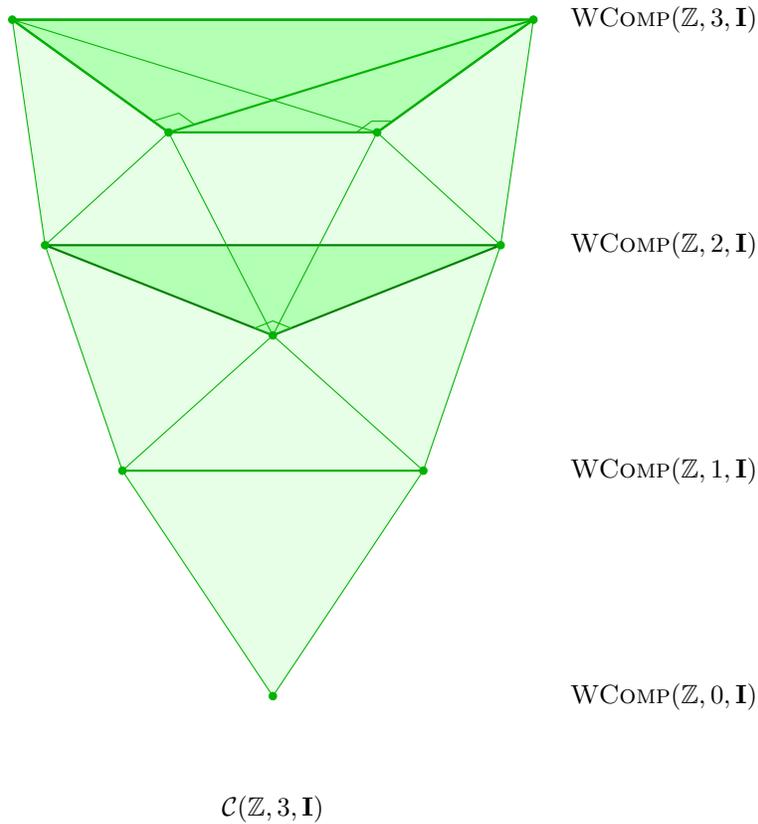
\begin{figure}
  \centering
  \begin{tikzpicture}
  \def\vsc{3cm}  
  \def\hscd{1.732cm} 
  \def\hscc{1.514cm} 
  \def\hscb{1cm} 

  \begin{scope}[xshift=0cm]
    \coordinate (d0) at (-2*\hscd,3*\vsc);
    \coordinate (dx) at (1.2*\hscd,-.5*\vsc);
    \coordinate (dy) at (1.6*\hscd,0*\vsc);
    \coordinate (dz) at (1.2*\hscd,.5*\vsc);
    \coordinate (dxy) at ([shift=(dy)]dx);
    \coordinate (dyz) at ([shift=(dz)]dy);
    \coordinate (d1) at ([shift=(dx)]d0);
    \coordinate (d2) at ([shift=(dy)]d1);
    \coordinate (d3) at ([shift=(dz)]d2);
    \coordinate (c0) at (-2*\hscc,2*\vsc);
    \coordinate (c1) at (0*\hscc,1.6*\vsc);
    \coordinate (c2) at (2*\hscc,2*\vsc);
    \coordinate (cx) at (2*\hscc,-.4*\vsc);
    \coordinate (cy) at (2*\hscc,.4*\vsc);
    \coordinate (b0) at (-2*\hscb,1*\vsc);
    \coordinate (b1) at (2*\hscb,1*\vsc);
    \coordinate (a0) at (0cm,0*\vsc);
    \coordinate (dl) at (3*\hscd,3*\vsc);
    \coordinate (cl) at (3*\hscd,2*\vsc);
    \coordinate (bl) at (3*\hscd,1*\vsc);
    \coordinate (al) at (3*\hscd,0*\vsc);
    \path (d0) -- (d1) node (e01) [pos=1/3]{};
    \path (d0) -- (d2) node (e02) [pos=1/3]{}; 
    \path (d0) -- (d3) node (e03) [pos=1/3]{}; 
    \path (d1) -- (d2) node (e12) [pos=1/3]{}; 
    \path (d1) -- (d3) node (e13) [pos=1/3]{}; 
    \path (d2) -- (d3) node (e23) [pos=1/3]{}; 
    \coordinate (E01) at (e01);
    \coordinate (E02) at (e02);
    \coordinate (E03) at (e03);
    \coordinate (E12) at (e12);
    \coordinate (E13) at (e13);
    \coordinate (E23) at (e23);
    \path (c0) -- (c1) node (f01) [pos=1/3]{};
    \path (c0) -- (c2) node (f02) [pos=1/3]{}; 
    \path (c1) -- (c2) node (f12) [pos=1/3]{}; 
    \coordinate (F01) at (f01);
    \coordinate (F02) at (f02);
    \coordinate (F12) at (f12);
    \path (b0) -- (b1) node (g01) [pos=1/3]{};
    \coordinate (G01) at (g01);
    
    \fill[fill = green!10] (d0)--(c0)--(b0)--(a0)--(b1)--(c2)--(d3)-- cycle;
    
    \filldraw[GreenPoly] (c0)--(c1)--(c2)--cycle;
    \filldraw[GreenPoly] (d0)--(d1)--(d2)--(d3)--cycle;
    \draw[color=green!30,ultra thick] (d1)--(d3);
    \draw[GreenLine] (d0)--(d1)--(d2)--(d3)--cycle;
    \draw[GreenLine] (b0)--(b1);

    \draw[thin,green!70!black] (d0)--(c0)--(d1)--(c1)--(d2)--(c2)--(d3);
    \draw[thin,green!70!black] (c0)--(b0)--(c1)--(b1)--(c2);
    \draw[thin,green!70!black] (b0)--(a0)--(b1);
    
    \newcommand{\drawrect}[6]{
      \draw[#4] (#1)-- ++([scale=-#5]#2)-- ++([scale=#6]#3)-- ++([scale=#5]#2) --cycle;
    }

    \drawrect{d1}{dx}{dyz}{GreenLine,thin}{.1}{.0707}
    \drawrect{d2}{dy}{dz}{GreenLine,thin}{.1}{.1}
    \drawrect{c1}{cx}{cy}{GreenLine,thin}{.08}{.08}
    
    \draw[GreenLine] (d1)--(d3);
    \draw[GreenLine,thin] (d0)--(d2);
    
    \def\d{1}

    \foreach \x in {d0,d1,d2,d3,c0,c1,c2,b0,b1,a0} {
      \node[dot,color=black!30!green, inner sep = \d pt] at (\x) {};
    }
    
    \node at (dl) {$\wcomp(\Z,3,\Ib)$};
    \node at (cl) {$\wcomp(\Z,2,\Ib)$};
    \node at (bl) {$\wcomp(\Z,1,\Ib)$};
    \node at (al) {$\wcomp(\Z,0,\Ib)$};
    \node at (0,-.5*\vsc) {$\mathcal{C}(\Z,3,\Ib)$};
  \end{scope}
  
\end{tikzpicture}
  \caption{In the graded topological poset $\Cc(\Z,3,\Ib)$, 
    the elements of rank $k$ have the metric topology of the standard 
    $k$-dimensional orthoscheme $\wcomp(\Z,k,\Ib)$.
    \label{fig:top-graded-poset}}
\end{figure}

We illustrate this construction with $G=\Z$ and $g=n=3$ since this is
one of the few examples where one can sketch the full poset.

\begin{example}\label{ex:subcompositions}
  When $G = \Z$ and $X = \{1\}$ we denote $\Fc(\Z,n,\Ib)$ by
  $\Cc(\Z,n,\Ib)$ and refer to subfactorizations as
  \emph{subcompositions}. The elements of rank $k$ in $\Cc(\Z,n,\Ib)$
  are precisely those in $\wcomp(\Z,k,\Ib)$ by definition. Each point
  is labeled by an element of $\wcomp(\Z,k,\Ib)$, i.e. a $k$-element
  multiset in the interval, and each cell is labeled by an element of
  $\comp(\Z,k,\Ib)$, the list of multiplicities in their natural linear
  order.  In particular, the elements in each rank form a standard
  metric orthoscheme.  See Figure~\ref{fig:top-graded-poset} for an
  illustration when $n = 3$.  Note that there is a stark asymmetry in
  this ordering.  For any element of rank $k$ with $0 < k < n$, there
  is only a finite number of elements below it in rank $k-1$
  (obtained by removing one of the elements of this $k$-element
  multiset), and there is a continuum of elements above it in rank
  $k+1$ (obtained by adding a random real number in $\Ib$ to the
  multiset).
\end{example}

For each $h \in P_g = [1,g]$, there is an $h'$ such that $h \cdot h' =
g$ and $[1\ h\ h'\ 1]$ is a linear factorization of $g$.  The ideal or
lower set $\low(h)$ of elements below $h$ in $P_g$ is isomorphic to
the interval $P_h$ and the filter or upper set $\upper(h)$ of elements
above $h$ in $P_g$ is isomorphic to $P_{h'}$.  In the linear
topological poset $\Fc(G,g,\Ib)$ (Definition~\ref{def:fact-top-poset}),
lower sets are also products of intervals inside $P_g$.

\begin{rem}[Lower sets]\label{rem:lower-sets}
  It is straightforward to compute the ideal or lower set of an
  element in $\Fc(G,g,\Ib)$. For example, if $\bu \in \Fc(G,g,\Ib)$
  with $\bu = 0^{x_L}s_1^{x_1}\cdots s_k^{x_k}1^{x_R}$, then the lower
  set $\low(\bu)$ is isomorphic to the direct product $P_{x_L} \times
  P_{x_1} \times \cdots \times P_{x_k} \times P_{x_R}$.  This is
  because each new exponent must be below the old exponent in the
  ordering of $P_g$ and these choices are independent.
\end{rem}

\begin{example}
  If $\bs = 0^{a_L}s_1^{a_1}\cdots s_k^{a_k}1^{a_R}$ is a multiset in
  $\Cc(\Z,n,\Ib)$, then its lower set $\low(\bs)$ is isomorphic to the
  product of chain posets with lengths $a_L, a_1,\ldots,a_k,a_R$.  See
  Figure~\ref{fig:lower-poset} for an illustration for the $5$ element
  multiset $\bs = 0^0 s^2 t^3 1^0$ with $0 < s < t < 1$. Note that
  when $\bs = 0^1s_1^1\cdots s_k^11^1$, this means that the lower set
  $\low(\bs)$ is isomorphic to the Boolean lattice $\bool(k+2)$.
\end{example}

\begin{figure}
  \centering
  \begin{tikzpicture}
    \begin{scope}
      \node[RowNode] (top) at (3,3) {$s^2 t^3$};
      
      \node[RowNode] (a1) at (2,1.5) {$s^1 t^2$};
      \node[RowNode] (a2) at (4,1.5) {$s^2 t^2$};
      
      \node[RowNode] (b1) at (1,0) {$s^0 t^3$};
      \node[RowNode] (b2) at (3,0) {$s^1 t^2$};
      \node[RowNode] (b3) at (5,0) {$s^2 t^1$};
      
      \node[RowNode] (c1) at (2,-1.5) {$s^0 t^2$};
      \node[RowNode] (c2) at (4,-1.5) {$s^1 t^1$};
      \node[RowNode] (c3) at (6,-1.5) {$s^2 t^0$};
      
      \node[RowNode] (d1) at (3,-3) {$s^0 t^1$};
      \node[RowNode] (d2) at (5,-3) {$s^1 t^0$};
      
      \node[RowNode] (bottom) at (4,-4.5) {$s^0 t^0$};
      
      \draw (top) -- (a1);
      \draw (top) -- (a2);
      
      \draw (a1) -- (b1);
      \draw (a1) -- (b2);
      \draw (a2) -- (b2);
      \draw (a2) -- (b3);
      
      \draw (b1) -- (c1);
      \draw (b2) -- (c1);
      \draw (b2) -- (c2);
      \draw (b3) -- (c2);
      \draw (b3) -- (c3);
      
      \draw (c1) -- (d1);
      \draw (c2) -- (d1);
      \draw (c2) -- (d2);
      \draw (c3) -- (d2);
      
      \draw (d1) -- (bottom);
      \draw (d2) -- (bottom);
    \end{scope}
  \end{tikzpicture}
  \caption{The lower set $\low(\bs)$ in $\Cc(\Z,5,\Ib)$ for the
    $5$-element multiset $\bs = 0^0 s^2 t^3 1^0$.  In the figure, the
    unused initial $0$ and final $1$ have been
    suppressed.}\label{fig:lower-poset}
\end{figure}
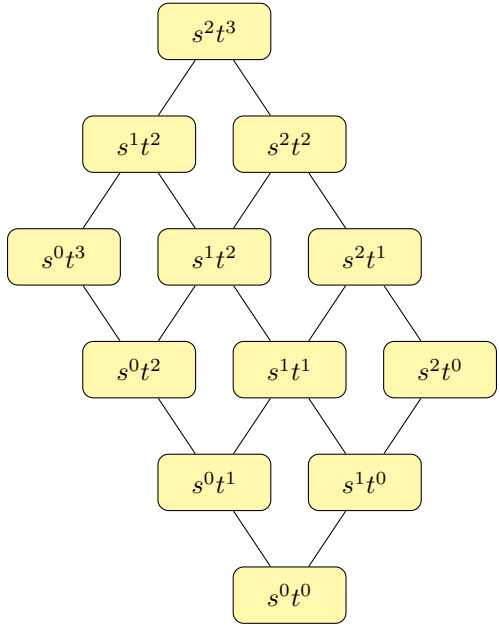

Since lower sets in $\Fc(G,g,\Ib)$ are products of subintervals in
$P_g$, these are discrete when $G$ is discrete and finite when $G$ is
finite. Upper sets in $\Fc(G,g,\Ib)$, on the other hand, are
uncountable, but they have a familiar structure.  To characterize
upper sets in $\Fc(G,g,\Ib)$ we define operations on multisets and
prove a technical lemma.

\begin{defn}[Multiset operations]\label{def:multiset-op}
  Given $G$-multisets $\bu\colon \Ib \to G$ and $\bv\colon \Ib \to G$,
  define the \emph{product $G$-multiset} $\bu\bv$ by the formula
  $(\bu\bv)(r) = \bu(r)\bv(r)$ for all $r \in \Ib$ and the
  \emph{inverse $G$-multiset} $\bu^{-1}$ by the formula $\bu^{-1}(r) =
  (\bu(r))^{-1}$ for all $r \in \Ib$.
\end{defn}

Note that if $\bu,\bv \in \Fc(G,g,\Ib)$ with $\bv \subseteq \bu$, then
$\bv^{-1}\bu$ is a $G$-multiset such that $\bv(\bv^{-1}\bu) = \bu$,
and by Lemma~\ref{lem:interval-factors}, $\bv^{-1}\bu$ is also an
element of $\Fc(G,g,\Ib)$.

\begin{lem}\label{lem:multiset-division-order}
  Let $\bu, \bv, \bw \in \Fc(G,g,\Ib)$ with $\bv \subseteq \bu$ and
  $\bv\subseteq \bw$.  Then $\bu \subseteq \bw$ if and only if
  $\bv^{-1}\bu \subseteq \bv^{-1}\bw$.
\end{lem}

\begin{proof}
  Since $\bv \subseteq \bu$ and $\bv \subseteq \bw$, we know that
  $\ell(\bv^{-1}(r)\bu(r)) = \ell(\bu(r)) - \ell(\bv(r))$ and
  $\ell(\bv^{-1}(r)\bw(r)) = \ell(\bw(r)) - \ell(\bv(r))$.  By
  definition, $\bv^{-1}\bu \subseteq \bv^{-1}\bw$ if and only if
  $\bv^{-1}(r)\bu(r) \leq \bv^{-1}(r)\bw(r)$ for all $r \in \Ib$,
  which is equivalent to saying that
  \[
  \ell(\bv^{-1}(r)\bu(r)) + \ell(\bu^{-1}(r)\bw(r)) = \ell(\bv^{-1}(r)\bw(r)).
  \]
  Plugging in, we find that $\ell(\bu(r)) + \ell(\bu^{-1}(r)\bw(r)) =
  \ell(\bw(r))$, which means that $\bu(r) \leq \bw(r)$ for all $r\in
  \Ib$, i.e. $\bu \subseteq \bw$.
\end{proof}

We now use Lemma~\ref{lem:multiset-division-order} to completely
characterize upper sets in $\Fc(G,g,\Ib)$.

\begin{thm}\label{thm:interval-lower-upper} 
  Let $\bv \in \Fc(G,g,\Ib)$ with $\rho(\bv) = h$ and let $h'$ be the
  element in $P_g$ with $h \cdot h' = g$.  The upper set $\upper(\bv)$
  is isomorphic to the topological poset $\Fc(G,h',\Ib)$.  As a
  consequence, the maximal elements in $\upper(\bv)$ form a subspace
  homeomorphic to the order complex $O_{h'} = \Delta(P_{h'})$.
\end{thm}

\begin{proof}
  For the first claim, note that for each $r \in I$, the product
  $\prod_{s>r}\bv(s)$, in which the factors are arranged left to right
  in increasing order of $s$, is a well-defined element of $P_g$ since
  there are only finitely many $s\in I$ such that $\bv(s)$ is
  nontrivial. Using this, we define $\phi \colon \upper(\bv) \to
  \Fc(G,h',\Ib)$ by
  \[\phi(\bu) = \left(\prod_{s>r} \bv(s) \right)^{-1} (\bv^{-1}\bu)(r)
  \left(\prod_{s>r} \bv(s) \right).\] If $\rho(\bu) = h''$, then we
  can see by definition that $\rho(\phi(\bu)) = h^{-1}h''$ and
  therefore we do indeed have $\phi(\bu) \in \Fc(G,h',\Ib) =
  \Fc(G,h^{-1}g,\Ib)$.  This function is a bijection with inverse
  $\phi^{-1}\colon \Fc(G,h^{-1}g,\Ib) \to \upper(\bv)$ given
  by \[\phi^{-1}(\bw) = \bv(r) \left(\prod_{s>r} \bv(s) \right) \bw(r)
  \left( \prod_{s>r} \bv(s) \right)^{-1}.\] By
  Lemma~\ref{lem:multiset-division-order} and the fact that
  conjugation preserves the partial order on $G$, both $\phi$ and
  $\phi^{-1}$ are order-preserving maps. Therefore, $\phi$ is a poset
  isomorphism. By Proposition~\ref{prop:wfact-order-complex}, the
  maximal elements in $\Fc(G,h',\Ib)$ form a subspace homeomorphic to
  the order complex $O_{h'} =\wfact(G,h',\Ib)$, proving the second
  claim.
\end{proof}

To describe this another way, we can write $\bu = \bv(\bv^{-1}\bu)$
and then deform $\bu$ by pushing the terms appearing from $\bv$ to the
left end of the interval, conjugating the elements of $\bv^{-1}\bu$
along the way to preserve the fact that the product is $h''$. The
weighted factorization obtained by applying these conjugations to
$\bv^{-1}\bu$ is what we call $\phi(\bu)$, and the effect on the group
elements is an example of a \emph{Hurwitz move}. Viewing this as a 
continuous deformation makes clear that $\phi$ is not just an 
isomorphism, but an isometry as well. 

\begin{example}
  Let $\bs = [0^0 s^1 1^0]$ be a multiset of size $1$ in the unit
  interval and an element of $\Cc(\Z,3,\Ib)$.  The upper set
  $\upper(\bs)$ is the collection of multisets of size at most $3$ in
  $\Ib$ which contains $s$ as one of its elements.  By
  Theorem~\ref{thm:interval-lower-upper}, this is a subposet
  isomorphic (as a topological poset) to $\Cc(\Z,2,\Ib)$.  This is
  illustrated in Figure~\ref{fig:top-graded-upset}.  The bending of
  the image under the embedding is caused by the inequalities between
  the one or two new multiset elements and the original multiset
  element $s$.
\end{example}

\begin{figure}
  \centering \begin{tikzpicture}
  \def\vsc{3cm}  
  \def\hscd{1.732cm} 
  \def\hscc{1.514cm} 
  \def\hscb{1cm} 

  \begin{scope}[xshift=-4cm]
    \coordinate (c0) at (-2*\hscc,3*\vsc);
    \coordinate (c1) at (0*\hscc,2.6*\vsc);
    \coordinate (c2) at (2*\hscc,3*\vsc);
    \coordinate (cx) at (2*\hscc,-.4*\vsc);
    \coordinate (cy) at (2*\hscc,.4*\vsc);
    
    \coordinate (b0) at (-2*\hscb,2*\vsc);
    \coordinate (b1) at (2*\hscb,2*\vsc);
    \coordinate (a0) at (0cm,1*\vsc);
    \path (c0)--(c1) node (f01) [pos=1/3]{};
    \path (c0)--(c2) node (f02) [pos=1/3]{}; 
    \path (c1)--(c2) node (f12) [pos=1/3]{}; 
    \coordinate (F01) at (f01);
    \coordinate (F02) at (f02);
    \coordinate (F12) at (f12);
    \path (b0)--(b1) node (g01) [pos=1/3]{};
    \coordinate (G01) at (g01);
    \coordinate (cl) at (-3*\hscd,3*\vsc);
    \coordinate (bl) at (-3*\hscd,2*\vsc);
    \coordinate (al) at (-3*\hscd,1*\vsc);

    
    \fill[fill = blue!10] (c0)--(b0)--(a0)--(b1)--(c2)-- cycle;
    
    \filldraw[BluePoly] (c0)--(c1)--(c2)--cycle;
    \draw[BlueLine,thin] (F01)--(F02)--(F12);
    \draw[BlueLine] (b0)--(b1);
    \draw[thin,blue] (c0)--(b0)--(c1)--(b1)--(c2);
    \draw[thin,blue] (b0)--(a0)--(b1);
    \draw[ultra thin, dashed, blue] (F01)--(b0) (c1)--(G01)--(a0) (F12)--(b1);

    \newcommand{\drawrect}[6]{
      \draw[#4] (#1)-- ++([scale=-#5]#2)-- ++([scale=#6]#3)-- ++([scale=#5]#2) --cycle;
    }

    \drawrect{c1}{cx}{cy}{BlueLine,thin}{.08}{.08}
    \drawrect{F01}{cx}{cy}{BlueLine,thin}{.08}{.08}
    \drawrect{F12}{cx}{cy}{BlueLine,thin}{.08}{.08}
    
    \def\d{1}
    \foreach \x in {c0,c1,c2,b0,b1,a0,F01,F02,F12,G01}{
      \node[dot,color=black!30!blue, inner sep = \d pt] at (\x) {};
    }
  \end{scope}

  \draw[|-latex,thick] (0*\hscb,2*\vsc)--(1*\hscb,2*\vsc);
  
  \begin{scope}[xshift=5cm]
    \coordinate (d0) at (-2*\hscd,3*\vsc);
    \coordinate (dx) at (1.2*\hscd,-.5*\vsc);
    \coordinate (dy) at (1.6*\hscd,0*\vsc);
    \coordinate (dz) at (1.2*\hscd,.5*\vsc);
    \coordinate (dxy) at (2.8*\hscd,-.5*\vsc);
    \coordinate (dyz) at (2.8*\hscd,.5*\vsc);
    \coordinate (d1) at (-.8*\hscd,2.5*\vsc);
    \coordinate (d2) at (.8*\hscd,2.5*\vsc);
    \coordinate (d3) at (2*\hscd,3*\vsc);
    \coordinate (c0) at (-2*\hscc,2*\vsc);
    \coordinate (c1) at (0*\hscc,1.6*\vsc);
    \coordinate (c2) at (2*\hscc,2*\vsc);
    \coordinate (cx) at (2*\hscc,-.4*\vsc);
    \coordinate (cy) at (2*\hscc,.4*\vsc);
    \coordinate (b0) at (-2*\hscb,1*\vsc);
    \coordinate (b1) at (2*\hscb,1*\vsc);
    \coordinate (a0) at (0cm,0*\vsc);
    \coordinate (dl) at (3*\hscd,3*\vsc);
    \coordinate (cl) at (3*\hscd,2*\vsc);
    \coordinate (bl) at (3*\hscd,1*\vsc);
    \coordinate (al) at (3*\hscd,0*\vsc);
    \path (d0) -- (d1) node (e01) [pos=1/3]{};
    \path (d0) -- (d2) node (e02) [pos=1/3]{}; 
    \path (d0) -- (d3) node (e03) [pos=1/3]{}; 
    \path (d1) -- (d2) node (e12) [pos=1/3]{}; 
    \path (d1) -- (d3) node (e13) [pos=1/3]{}; 
    \path (d2) -- (d3) node (e23) [pos=1/3]{}; 
    \coordinate (E01) at (e01);
    \coordinate (E02) at (e02);
    \coordinate (E03) at (e03);
    \coordinate (E12) at (e12);
    \coordinate (E13) at (e13);
    \coordinate (E23) at (e23);
    \path (c0) -- (c1) node (f01) [pos=1/3]{};
    \path (c0) -- (c2) node (f02) [pos=1/3]{}; 
    \path (c1) -- (c2) node (f12) [pos=1/3]{}; 
    \coordinate (F01) at (f01);
    \coordinate (F02) at (f02);
    \coordinate (F12) at (f12);
    \path (b0) -- (b1) node (g01) [pos=1/3]{};
    \coordinate (G01) at (g01);
    
    \fill[fill = green!10] (d0)--(c0)--(b0)--(a0)--(b1)--(c2)--(d3)-- cycle;
    
    \filldraw[GreenPoly] (c0)--(c1)--(c2)--cycle;
    \fill[fill = blue!10]
    (E01)--(F01)--(G01)--(F12)--(E23)--(E03)--(E01)-- cycle;
    \draw[ultra thin, dashed, blue] (F02)--(G01);
    \filldraw[GreenPoly] (F01)--(F02)--(F12)--(c1)--cycle;
    \draw[thin, blue] (E01)--(F01)--(E12)--(F12)--(E23);
    \draw[ultra thin, dashed, blue] (E02)--(F01) (E12)--(F02) (E13)--(F12);
    \draw[thin, blue] (F01)--(G01)--(F12);
    \filldraw[GreenPoly] (d0)--(d1)--(d2)--(d3)--cycle;
    \draw[color=green!30,ultra thick] (d1)--(d3);
    \draw[GreenLine] (d0)--(d1)--(d2)--(d3)--cycle;
    \draw[GreenLine] (b0)--(b1);

    \draw[thin,green!70!black] (d0)--(c0)--(d1)--(c1)--(d2)--(c2)--(d3);
    \draw[thin,green!70!black] (c0)--(b0)--(c1)--(b1)--(c2);
    \draw[thin,green!70!black] (b0)--(a0)--(b1);

    \newcommand{\drawrect}[6]{
      \draw[#4] (#1)-- ++([scale=-#5]#2)-- ++([scale=#6]#3)-- ++([scale=#5]#2) --cycle;
    }

    \drawrect{d1}{dx}{dyz}{GreenLine,thin}{.1}{.0707}
    \drawrect{d2}{dy}{dz}{GreenLine,thin}{.1}{.1}
    \drawrect{c1}{cx}{cy}{GreenLine,thin}{.08}{.08}
    
    \filldraw[BluePoly] (E01)--(E03)--(E23)--(E13)--(E12)--(E02)--(E01);
    \draw[BlueLine,thin] (E02)--(E03)--(E13);
    \draw[BlueLine] (F01)--(F02)--(F12);
    \draw[GreenLine] (d1)--(d3);
    \draw[GreenLine,thin] (d0)--(d2);
    
    \foreach \x in {E01,E02,E03,E12,E13,E23,F01,F02,F12,G01}{
      \node[dot,color=black!30!blue, inner sep = .7 pt] at (\x) {};
    }
    \foreach \x in {d0,d1,d2,d3,c0,c1,c2,b0,b1,a0}{
      \node[dot,color=black!30!green, inner sep = 1 pt] at (\x) {};
    }
    

    \drawrect{E12}{dx}{dz}{BlueLine,thin}{.08}{.08}
    \drawrect{E02}{dy}{dz}{BlueLine,thin}{.08}{.08}
    \drawrect{E13}{dx}{dy}{BlueLine,thin}{.08}{.08}
  \end{scope}
  
\end{tikzpicture}
  \caption{In $\Cc(\Z,3,\Ib)$, the upper set of a rank-$1$ element
    (i.e. a multiset of size $1$ in the unit interval) is a copy of
    $\Cc(\Z,2,\Ib)$. The bending of the image is caused by the inequalities
    between the new multiset elements and the original multiset
    element.} \label{fig:top-graded-upset}
\end{figure}
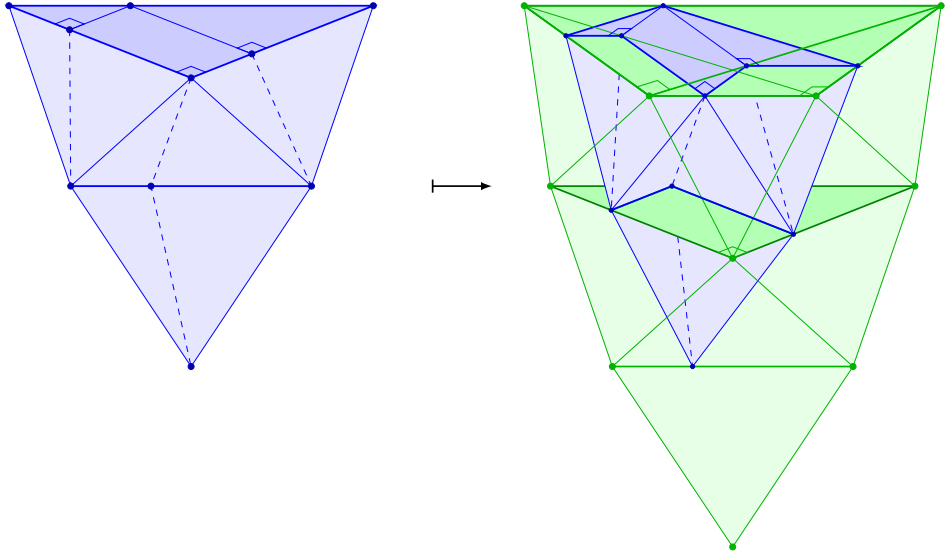


We now turn our attention to the more complicated circular version.

\begin{defn}[$\Fc(G,g,\Sb)$]\label{def:fact-top-poset-circ}
  Define an equivalence relation on $\Fc(G,g,\Ib)$ by declaring $\bu =
  0^{x_L} s_1^{x_1} \cdots s_k^{x_k} 1^{x_R}$ and $\bv = 0^{y_L}
  s_1^{y_1} \cdots s_k^{y_k} 1^{y_R}$ to be equivalent if $x_i = y_i$ for all $i \in
  \{1,\ldots,k\}$ and $h=gx_Rg^{-1}x_L = gy_Rg^{-1}y_L$. We denote the
  set of all such equivalence classes $[\bu]$ by $\Fc(G,g,\Sb)$ and
  observe that it inherits the partial order by subfactorizations from
  $\Fc(G,g,\Ib)$. In particular, each equivalence class $[\bu]$ can be
  uniquely represented by an element of the form $\bw = 0^{z_L}
  s_1^{z_1} \cdots s_k^{z_k} 1^{1}$, and applying $\ell \circ \rho$ to
  these representatives provides a rank function for $\Fc(G,g,\Sb)$.
  Finally, note that the quotient map $q \colon \Fc(G,g,\Ib) \to
  \Fc(G,g,\Sb)$ given by $q(0^{x_L} s_1^{x_1} \cdots s_k^{x_k}
  1^{x_R}) = 0^{h} s_1^{x_1} \cdots s_k^{x_k} 1^{1}$ is an
  order-preserving surjection.
\end{defn}

\begin{rem}\label{rem:not-Kh-union}
  Caution is required when working with $\Fc(G,g,\Sb)$. It is quite
  possible that $\bu$ and $\bv$ belong to the same equivalence class
  even when $\rho(\bu) \neq \rho(\bv)$. To give a small example, the
  weighted linear factorizations $[1\ a]$ and $[gag^{-1}\ 1]$ belong
  to the same equivalence class, but will have different values of
  $\rho$ if $a$ and $g$ do not commute.  In particular, the identified
  factorizations may belong to distinct order complexes
  $O_{\rho(\bu)}$ and $O_{\rho(\bv)}$.  This is why the rank function
  needed to be defined carefully.  And this also shows that the clean
  definition of $\Fc(G,g,\Ib)$ as a disjoint union of spaces $O_h =
  \wfact(G,h,\Ib)$ with $h \in P_g$, becomes more complicated in the
  circular version $\Fc(G,g,\Sb)$.  The elements in $\Fc(G,g,\Sb)$
  cannot be partitioned into a disjoint union of the spaces
  $K_h=\wfact(G,h,\Sb)$, one for each $h$ in $P_g$.
\end{rem}

An example of this gluing phenomenon already occurs in $\sym_3$.

\begin{example}\label{ex:subfact-delta}
  Let $G = \sym_3$, $\delta = (1\ 2\ 3)$ with $\delta = ab = bc = ca$
  as in Example~\ref{ex:absolute-order-interval}. The maximal elements
  (of rank~$2$) in the linear topological poset
  $\Fc(\sym_3,\delta,\Ib)$ form a subspace homeomorphic to $O_\delta$,
  the three isosceles right triangles with a common hypotenuse shown
  in Figure~\ref{fig:weighted-fact-2-interval}.  The rank~$1$ elements
  are three disjoint unit intervals, $O_a = \wfact(\sym_3,a,\Ib)$,
  $O_b = \wfact(\sym_3,b,\Ib)$ and $O_c = \wfact(\sym_3,c,\Ib)$, and
  the unique rank $0$ element is the complex $O_1 =
  \wfact(\sym_3,1,\Ib)$.  Under the quotient identification defined
  above, the rank~$2$ elements $\Fc(\sym_3,\delta,\Sb)$ form the $1$
  vertex interval complex $K_\delta$ schematically shown in
  Figure~\ref{fig:weighted-fact-2-circle}.  But the three unit
  intervals in rank~$1$ are identified end-to-end to form a single
  circle of length~$3$, rather than three disjoint circles $K_a \sqcup
  K_b \sqcup K_c$ each of length~$1$.  This is because conjugating by
  $\delta$ sends $a$ to $b$, $b$ to $c$ and $c$ to $a$.  
\end{example}

The gluing described in Example~\ref{ex:subfact-delta} is caused by
the orbit of $h \in P_g$ under the iterated conjugation by $g$.  In
general, if $k$ is the smallest positive integer such that $g^k$
commutes with $h$, then the $k$ subspaces $O_{h_i}$ with $h_i = g^i h
g^{-i}$ glue together to form a single cyclic $k$-sheeted cover of
  $K_h$.  Note that when $g$ and $h$ commute, we get a copy of $K_h$.  
  For example, $\Cc(\Z,n,\Sb)$ is a disjoint union of the spaces $K_i$ 
  for $i \in \{0,1,\ldots, n\}$ since $\mathbb{Z}$ is abelian.
  This also applies to the maximal elements of $\Fc(G,g,\Sb)$, since $g$
  commutes with $g$, and the maximal elements form a subspace
  homeomorphic to the interval complex $K_g = \wfact(G,g,\Sb)$.  In
  particular, we obtain Theorem~\ref{mainthm:max-elements} as a
  rephrasing of Proposition~\ref{prop:wfact-interval-complex}.

\begin{thm}[Theorem~\ref{mainthm:max-elements}]
  \label{thm:max-elements}
  Let $P_g$ be an interval poset in a marked group $G$ with order
  complex $O_g$ and interval complex $K_g$.  In the circular
  topological poset $\Fc(G,g,\Sb)$, the subspace of maximal elements
  is homeomorphic to $K_g$.
\end{thm}

\begin{proof}
  The only thing to check is that the quotient map defined in
  Definition~\ref{def:fact-top-poset-circ} agrees with the quotient
  map defined in Definition~\ref{def:weighted-fact-circle} when
  restricted to maximal elements, and this is straightforward.
\end{proof}

As an immediate corollary we have Theorem~\ref{mainthm:dual-complex}.

\begin{cor}[Theorem~\ref{mainthm:dual-complex}]
  \label{cor:dual-complex}
  Let $\cb$ be a Coxeter element in a Coxeter group $W$ generated by
  all reflections.  The maximal elements of the circular topological
  poset $\Fc(W,\cb,\Sb)$ form a subspace homeomorphic to the interval complex 
  $K_\cb$ whose fundamental group is the dual Artin
  group $\art^\ast(W,\cb)$.  In particular, the subspace of maximal
  elements in $\Fc(\sym_d,\delta,\Sb)$ is homeomorphic to the dual braid complex
  $K_\delta$ whose fundamental group is the braid group $\braid_d$.
\end{cor}

Once we prove Theorem~\ref{mainthm:nc-fact-isomorphism}
in Section~\ref{sec:cont-nc-part}, the topological poset $\Fc(\sym_d,\delta,\Sb)$
can be replaced with $\NCS{d}{}$.
And Theorem~\ref{mainthm:upper-sets} follows quickly from
Theorem~\ref{thm:interval-lower-upper}.

\begin{cor}[Theorem~\ref{mainthm:upper-sets}]\label{cor:circle-upper-sets}
  Let $\bu$ be a weighted linear factorization of $h \in P_g$ with
  equivalence class $[\bu] \in \Fc(G,g,\Sb)$. Then the upper set
  $\upper([\bu])$ in $\Fc(G,g,\Sb)$ is isomorphic to $\Fc(G,h',\Sb)$
  where $h\cdot h' = g$. Consequently, the maximal elements of
  $\upper([\bu])$ form a subspace which is isometric to the interval
  complex $K_{h'}$ inside the interval complex $K_g$.
\end{cor}

\begin{proof}
  By definition of the subfactorization order on $\Fc(G,g,\Sb)$ and
  the associated quotient map, we have $\upper([\bv]) =
  q(\upper(\bv))$, where $\bv \in \Fc(G,g,\Ib)$ is any representative
  of the equivalence class $[\bv]$. Thus, the first claim follows from
  Theorem~\ref{thm:interval-lower-upper}.  The second, following
  similar reasoning to the proof of Theorem~\ref{thm:interval-lower-upper}, 
  follows from Proposition~\ref{prop:wfact-interval-complex}.
\end{proof}

We conclude by discussing these results for our two running examples.

\begin{example}
  Let $\delta$ be the $d$-cycle $(1\ \cdots\ d) \in \sym_d$ and let
  $[\bu]$ be an equivalence class of a weighted linear factorization
  of $\gamma \in P_\delta = [1,\delta]$. 
  Then the upper set
  $\upper([\bu])$ in $\Fc(\sym_d,\delta,\Sb)$ is isomorphic to
  $\Fc(\sym_d,\gamma^{-1}\delta,\Sb)$. Moreover, the maximal elements
  of $\upper([\bu])$ form a subspace of the dual braid complex which
  is isometric to a product of dual braid complexes with smaller
  dimension.
\end{example}

\begin{example}
  Let $\bs \in \Cc(\Z,n,\Ib)$ with $\rho(\bs) = k$. Then $\upper(\bs)$
  is isomorphic to $\Cc(\Z,n-k,\Ib)$, and the maximal elements of
  $\upper(\bs)$ form a subspace homeomorphic to an orthoscheme of
  dimension $n-k$. Similarly, the upper set $\upper([\bs])$ in
  $\Cc(\Z,n,\Sb)$ is homeomorphic to $\Cc(\Z,n-k,\Sb)$, and the
  maximal elements of $\upper([\bs)$ form a subspace which is
    isometric to the quotient of the standard orthoscheme of dimension
    $n-k$ described in Example~\ref{ex:weighted-comp-circle}. 
\end{example}

\section{Continuous Noncrossing Partitions}\label{sec:cont-nc-part}

We now examine $\Fc(G,g,\Sb)$ in the special case when $G = \sym_d$,
$X=T$, and $g = \delta = (1\ 2\ \cdots\ d)$ as described in
Example~\ref{ex:absolute-order-interval}. In particular, we introduce
a new type of noncrossing partition and use this to prove
Theorems~\ref{mainthm:nc-fact-isomorphism} and
\ref{mainthm:dual-braid-spine}.

\begin{defn}[Noncrossing partitions]
  Let $P$ be a subset of the complex plane and let $\Pi(P)$ denote the
  lattice of all partitions of $P$, partially ordered by refinement.
  The elements of each partition are subsets of $P$ called blocks, and
  we say that a partition is \emph{noncrossing} if the convex hulls of
  its blocks are pairwise disjoint regions in $\mathbb{C}$. We define
  the \emph{poset of noncrossing partitions for $P$} to be the
  subposet of noncrossing elements in $\Pi(P)$, denoted $\NC(P)$.
\end{defn}

When $P$ is the vertex set for a convex $n$-gon, $\NC(P)$ is the
\emph{classical lattice of noncrossing partitions} $\NC(n)$,
originally defined by Kreweras in 1972 \cite{kreweras72}. The
following theorem, proven 25 years later by Biane, illustrates our
interest in $\NC(n)$.

\begin{thm}[\cite{biane}]\label{thm:nc-part-perm} 
  Let $\psi\colon \sym_d \to \Pi(d)$ be the function which sends each
  permutation to the partition formed by the orbits of its action on
  $\{1,\ldots,d\}$. Then $\psi$ restricts to an isomorphism from
  $P_\delta = [1,\delta]$ to $\NC(d)$.
\end{thm}

The decade following Biane's theorem produced a flurry of connections
between the absolute order on the symmetric group and the
combinatorics of noncrossing partitions---see the survey articles
\cite{baumeister19} and \cite{mccammond06} for more background.  One
of these connections, involving a variation on noncrossing partitions
due to Armstrong \cite{armstrong09}, will be of use to us later in
this section.

\begin{defn}[Shuffle partitions]\label{def:shuffle-partitions}
  Let $\pi \in \Pi(dk)$. Then $\pi$ is a \emph{$k$-shuffle partition}
  if $a \equiv b\ (\text{mod}\ k)$ whenever $a$ and $b$ belong to the
  same block of $\pi$.  When this is the case, note that for each $j
  \in \{1,\ldots,k\}$, we may identify the set $\{1,\ldots,d\}$ with
  the equivalence class of $j$ via the map $m \mapsto (m-1)k+j$ to
  obtain a partition $\pi_j \in \Pi(d)$ which is induced by
  $\pi$. Then $\pi$ is uniquely determined by the $k$-tuple
  $(\pi_1,\ldots,\pi_k)$.
\end{defn}

\begin{thm}[\protect{\cite[Theorem~4.3.5]{armstrong09}}]
  \label{thm:shuffle-reduced-product}
  Let $x_1,\ldots,x_k \in [1,\delta]$, let $\pi_i = \psi(x_i)$ for
  each $i$, and define $\pi$ to be the $k$-shuffle partition in
  $\Pi(dk)$ which is determined by the $k$-tuple
  $(\pi_1,\ldots,\pi_k)$. Then $\pi$ is noncrossing if and only if
  $\ell(x_1) + \cdots + \ell(x_k) = \ell(x_1\cdots x_k)$.
\end{thm}

Some recent progress on the structure of $\NC(P)$ when $P$ is finite
(but not convex) can be found in \cite{cdhm23}, but seemingly little
attention has been devoted to cases where $P$ is infinite. Here, we
are interested in the case where $P$ is the unit circle $\Sb$.  We
refer to $\NCS{}{}$ as the \emph{poset of continuous noncrossing
partitions}, and we are particularly interested in a subposet of these
which are compatible with a covering map for the circle.

\begin{defn}[Degree-$d$ partitions]\label{def:degree-d-invariant}
    Let $f\colon \Sb \to \Sb$ be the standard degree-$d$ covering map
    $f(z) = z^d$. We say that a partition $\pi \in \Pi(\Sb)$ is a
    \emph{degree-$d$ partition} if $f(z) = f(w)$ whenever $z$ and $w$
    belong to the same block of $\pi$. When this is the case, we may fix
    numbers $0 = s_L < s_1 < \cdots < s_k < 1$ such that $z$ belongs to
    a nontrivial block of $\pi$ only if $f(z) = e^{2\pi i s_j}$ for some
    $j$, then identify each preimage $f^{-1}(e^{2\pi i s_j})$ with the
    set $\{1,\ldots,d\}$ by reading off the preimages in increasing
    order of argument in $[0,2\pi)$. If we let $\pi_j$ be the partition
    in $\Pi(d)$ determined by $\pi$ in this way, then $\pi$ can be
    uniquely denoted by the expression $\pi = 0^{\pi_L} s_1^{\pi_1}
    \cdots s_k^{\pi_k}$.  It is clear from this construction that
    $\pi$ is noncrossing if and only if the $(k+1)$-shuffle partition
    determined by $(\pi_L,\pi_1,\ldots,\pi_k)$ is noncrossing.  Denote
    the subposet of all degree-$d$ partitions by $\Pi_d(\Sb)$ and the
    subposet of all noncrossing degree-$d$ partitions by
    $\NCS{d}{}$. Note that one may replace $f$ with any covering map
    $\Sb \to \Sb$ of degree $d$, and the resulting poset of
    ``$f$ noncrossing'' partitions will be isomorphic to
    $\NCS{d}{}$.
\end{defn}

\begin{example}\label{ex:ncs-k-shuffle}
  The degree-$12$ noncrossing partition in
  Figure~\ref{fig:ncs-example} can be described by the shorthand $\pi
  = 0^{\pi_L} 0.1^{\pi_1} 0.4^{\pi_2} 0.5^{\pi_3} 0.9^{\pi_4}$, where
  \begin{align*}
    \pi_L &= \{\{1\},\{2\},\{3\},\{4\},\{5\},\{6\},\{7\},\{8\},\{9\},\{10\},\{11\},\{12\}\}; \\
    \pi_1 &= \{\{1\},\{2\},\{3\},\{4,5\},\{6\},\{7\},\{8\},\{9\},\{10\},\{11\},\{12\}\}; \\
    \pi_2 &= \{\{1,2,3,11\},\{4\},\{5\},\{6\},\{7\},\{8\},\{9\},\{10\},\{12\}\}; \\
    \pi_3 &= \{\{1\},\{2\},\{3,6,8\},\{4\},\{5\},\{7\},\{9\},\{10\},\{11,12\}\}; \\
    \pi_4 &= \{\{1\},\{2\},\{3\},\{4\},\{5\},\{6\},\{7\},\{8,9,10\},\{11\},\{12\}\}.
  \end{align*}
  See Figure~\ref{fig:ncs-k-shuffle} for an illustration (omitting the discrete partition $\pi_L$).
  Note that when superimposed in the proper order, these partitions form a 
  noncrossing hypertree on the vertex set $\{1,\ldots,12\}$ as described in \cite{Mc-ncht}.
\end{example}

\begin{figure}
  \centering

\tikzstyle{dot}=[shape=circle,draw,color=black,fill=black,inner sep=1.3pt]

\begin{tikzpicture}
  \def\ra{3cm}
  \def\mid#1#2{(#2-#1)/2}
  \def\Arc#1#2#3{arc ({#1+270}:{#2+90}:{tan(\mid{#1}{#2})*#3})}
  \def\drawArc#1#2#3#4{\draw[#1] (#2:#4) \Arc{#2}{#3}{#4}}
      
    \begin{scope}[shift = {(0,0)}, line width = 0.35mm]
        \filldraw[color=black,fill=yellow!20, line width = 0.2mm] (0,0) circle (\ra);

        \drawArc{color=red}{93}{123}{\ra}; 

        \filldraw[color=green!50!black,fill=green!30] (12:\ra) 
            \Arc{12}{42}{\ra} 
            \Arc{42}{72}{\ra}
            \Arc{72}{312}{\ra}
            \Arc{312}{372}{\ra};
        
        \drawArc{color=blue}{315}{345}{\ra};  

        \filldraw[color=blue,fill=blue!20] (75:\ra) 
            \Arc{75}{165}{\ra}
            \Arc{165}{225}{\ra}
            \Arc{225}{435}{\ra};

        \filldraw[color=red!60!blue!80,fill=red!40!blue!50] (237:\ra) 
            \Arc{237}{267}{\ra}
            \Arc{267}{297}{\ra}
            \Arc{297}{597}{\ra};

        \foreach \i in {1,...,12} {
            \node[dot,color=red] () at (\i*30+3:\ra) {};
            \node[dot,color=black!30!green] () at (\i*30+12:\ra) {};
            \node[dot,color=blue] () at (\i*30+15:\ra) {};
            \node[dot,color=red!60!blue!80] () at (\i*30+27:\ra) {};
        }
    \end{scope}

    \begin{scope}[shift = {(-7.5,-5)}, line width = 0.35mm]
        \filldraw[color=black,fill=yellow!20, line width = 0.2mm] (0,0) circle (\ra/2);
        \drawArc{color=red}{90}{120}{\ra/2};
        \foreach \i in {1,...,12} {
            \node[dot,color=red] (v\i) at (\i*30-30:\ra/2) {};
            \node () at (\i*30-30:\ra/2+3mm) {$\i$};
        }
        \node () at (0,-2.5) {\Large $\pi_1$};
    \end{scope}

    \begin{scope}[shift = {(-2.5,-5)}, line width = 0.35mm]
        \filldraw[color=black,fill=yellow!20, line width = 0.2mm] (0,0) circle (\ra/2);
        \filldraw[color=green!50!black,fill=green!30] (0:\ra/2) 
            \Arc{0}{30}{\ra/2} 
            \Arc{30}{60}{\ra/2}
            \Arc{60}{300}{\ra/2}
            \Arc{300}{360}{\ra/2};
        \foreach \i in {1,...,12} {
            \node[dot,color=black!30!green] (v\i) at (\i*30-30:\ra/2) {};
            \node () at (\i*30-30:\ra/2+3mm) {$\i$};
        }
        \node () at (0,-2.5) {\Large $\pi_2$};
    \end{scope}

    \begin{scope}[shift = {(2.5,-5)}, line width = 0.35mm]
        \filldraw[color=black,fill=yellow!20, line width = 0.2mm] (0,0) circle (\ra/2);
        \drawArc{color=blue}{300}{330}{\ra/2};  
        \filldraw[color=blue,fill=blue!20] (60:\ra/2) 
            \Arc{60}{150}{\ra/2}
            \Arc{150}{210}{\ra/2}
            \Arc{210}{420}{\ra/2};
        \foreach \i in {1,...,12} {
            \node[dot,color=blue] (v\i) at (\i*30-30:\ra/2) {};
            \node () at (\i*30-30:\ra/2+3mm) {$\i$};
        }
        \node () at (0,-2.5) {\Large $\pi_3$};
    \end{scope}

    \begin{scope}[shift = {(7.5,-5)}, line width = 0.35mm]
        \filldraw[color=black,fill=yellow!20, line width = 0.2mm] (0,0) circle (\ra/2);
        \filldraw[color=red!60!blue!80,fill=red!40!blue!50] (210:\ra/2)
            \Arc{210}{240}{\ra/2}
            \Arc{240}{270}{\ra/2}
            \Arc{270}{570}{\ra/2};
        \foreach \i in {1,...,12} {
            \node[dot,color=red!60!blue!80] (v\i) at (\i*30-30:\ra/2) {};
            \node () at (\i*30-30:\ra/2+3mm) {$\i$};
        }
        \node () at (0,-2.5) {\Large $\pi_4$};
    \end{scope}

\end{tikzpicture}

  \caption{A degree-$12$-invariant noncrossing partition $\pi$
    together with its component partitions $\pi_1,\pi_2,\pi_3,\pi_4
    \in \NC(d)$, as described in
    Example~\ref{ex:ncs-k-shuffle}.}
  \label{fig:ncs-k-shuffle}
\end{figure}
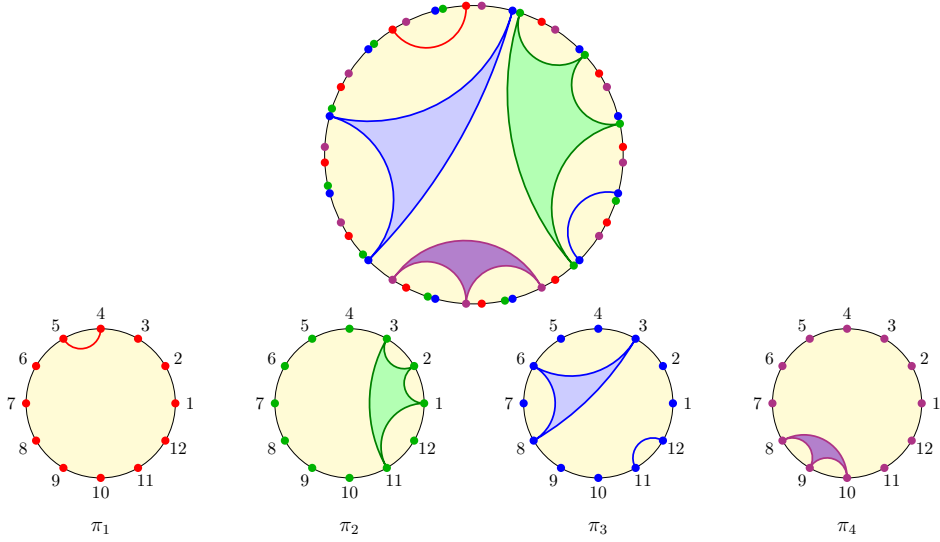

It would have been reasonable to refer to the elements of $\NCS{d}{}$
as ``continuous shuffle partitions'' considering their resemblance to
Definition~\ref{def:shuffle-partitions}. Instead, we use the
descriptor \emph{degree-$d$ noncrossing partitions} to recognize their
previous appearance in unpublished work of the late W. Thurston, which
was later completed by Baik, Gao, Hubbard, Lei, Lindsey, and
D. Thurston \cite{thurston20}. In the article, Thurston and his
collaborators described a spine for the space of monic complex
polynomials with $d$ distinct roots, where each point is labeled by a
``primitive major'' of a ``degree-$d$-invariant'' lamination of the
disk.  We have adopted the term ``degree-$d$'' but avoided the word
``invariant'' since these partitions are rarely invariant under a 
$2\pi/d$ rotation.

The following definition and lemma rephrase a useful observation of
Thurston's regarding the maximum number of non-singleton blocks in a
degree-$d$ noncrossing partition.

\begin{defn}[Total criticality]\label{def:total-criticality}
  Let $\pi \in \Pi_d(\Sb)$ and suppose that $\pi$ has $k$
  non-singleton blocks, denoted $A_1,\ldots,A_k$. The \emph{total
  criticality} of the partition $\pi$ is defined to be $|A_1| + \cdots
  + |A_k| - k$.  It is straightforward to see that the total
  criticality gives a rank function $\rank\colon \Pi_d(\Sb) \to
  \mathbb{N}$, and this descends to a rank function for $\NCS{d}{}$.
\end{defn}

The total criticality of a generic degree-$d$ partition can be
arbitrarily large, but this is not the case for their noncrossing
counterparts.

\begin{lem}[\protect{\cite[Proposition~2.1]{thurston20}}]
  \label{lem:total-criticality}
  If $\pi \in \NCS{d}{}$, then the total criticality of $\pi$ is at
  most $d-1$. Consequently, $\NCS{d}{}$ is a graded poset of height
  $d-1$.
\end{lem}

It follows from Lemma~\ref{lem:total-criticality} that a degree-$d$
noncrossing partition has at most $d-1$ non-singleton blocks (in which
case each has two elements). We can also give a clear characterization
of the maximal elements in $\NCS{d}{}$ as follows.

\begin{defn}[Complementary regions]
  Let $\pi \in \NCS{d}{}$ and identify $\Sb$ with the boundary of a
  disk.  The \emph{complementary regions} of $\pi$ are the connected
  components of the disk after removing the convex hulls of the blocks
  of $\pi$.
\end{defn}

\begin{lem}\label{lem:maximal-elements}
  The partition $\pi \in \NCS{d}{}$ has $\rank(\pi)+1$ complementary
  regions.  Consequently, the maximal elements of $\NCS{d}{}$ are
  those which have exactly $d$ complementary regions.
\end{lem}

\begin{proof}
  If $\pi$ is a degree-$d$ noncrossing partition, then we can
  illustrate $\pi$ in the disk and define a dual bipartite tree for
  $\pi$ by placing a black vertex in each convex hull and a white
  vertex in each complementary region, then connecting a black vertex
  to a white vertex when the two corresponding regions are
  adjacent. If $\pi$ has $k$ non-singleton blocks, then the tree must
  have $k$ black vertices and $\rank(\pi) + k = d-1 + k$ edges. By the
  Euler characteristic, the number of vertices for a tree is always
  one more than the number of edges, so it follows that the number of
  white vertices, and thus the number of complementary regions, is
  $\rank(\pi)+1$. By Lemma~\ref{lem:total-criticality}, we may thus
  conclude that $\pi$ is a maximal element of $\NCS{d}{}$ if and only
  if it has $d$ complementary regions.
\end{proof}

Given $\pi \in \NCS{d}{}$, we can draw the convex hulls of the blocks
of $\pi$ in the disk with unit circumference (rather than unit
radius), then deformation retract each convex hull to a point. Under
this transformation, the boundary of the disk becomes a metric graph
known as a \emph{cactus}, and in the special case where $\pi$ is
maximal, Lemma~\ref{lem:maximal-elements} tells us that this graph can
be built by gluing together $d$ circles, each of length $1/d$.  These
metric graphs make an appearance in \cite{gcp1}, where we associate
such a graph to each complex polynomial with distinct roots and
critical values on the unit circle. In \cite[Section~8]{gcp1}, we also
described a connection between continuous noncrossing partitions (then
referred to as ``real'' noncrossing partitions) and a type of metric
tree that we called a \emph{banyan}.  These connections between
$\NCS{d}{}$ and complex polynomials are of central importance in our
ongoing paper series \cite{gcp2}.

We are now ready to prove Theorem~\ref{mainthm:nc-fact-isomorphism}.
Recall from Definition~\ref{def:g-multisets} that the set of all
$G$-multisets from $S$ to $G$ is denoted by $\mult(G,S)$.

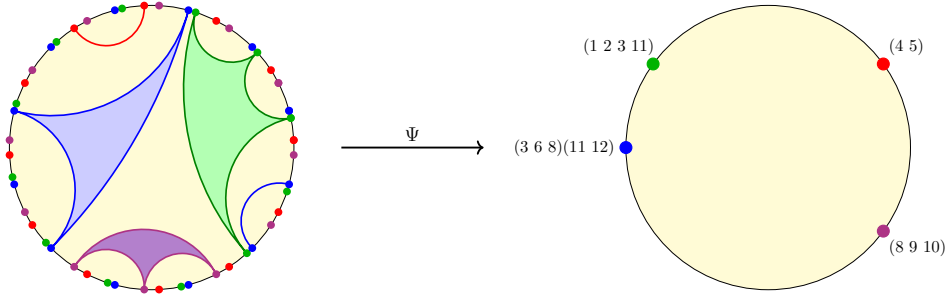
\begin{figure}
  \centering

\tikzstyle{dot}=[shape=circle,draw,color=black,fill=black,inner sep=1.3pt]

\begin{tikzpicture}
  \def\ra{3cm}
  \def\mid#1#2{(#2-#1)/2}
  \def\Arc#1#2#3{arc ({#1+270}:{#2+90}:{tan(\mid{#1}{#2})*#3})}
  \def\drawArc#1#2#3#4{\draw[#1] (#2:#4) \Arc{#2}{#3}{#4}}
      
    \begin{scope}[shift = {(-6,0)}, line width = 0.35mm]
        \filldraw[color=black,fill=yellow!20, line width = 0.2mm] (0,0) circle (\ra);

        \drawArc{color=red}{93}{123}{\ra}; 

        \filldraw[color=green!50!black,fill=green!30] (12:\ra) 
            \Arc{12}{42}{\ra} 
            \Arc{42}{72}{\ra}
            \Arc{72}{312}{\ra}
            \Arc{312}{372}{\ra};
        
        \drawArc{color=blue}{315}{345}{\ra};  

        \filldraw[color=blue,fill=blue!20] (75:\ra) 
            \Arc{75}{165}{\ra}
            \Arc{165}{225}{\ra}
            \Arc{225}{435}{\ra};

        \filldraw[color=red!60!blue!80,fill=red!40!blue!50] (237:\ra) 
            \Arc{237}{267}{\ra}
            \Arc{267}{297}{\ra}
            \Arc{297}{597}{\ra};

        \foreach \i in {0,...,12} {
            \node[dot,color=red] () at (\i*30+3:\ra) {};
            \node[dot,color=black!30!green] () at (\i*30+12:\ra) {};
            \node[dot,color=blue] () at (\i*30+15:\ra) {};
            \node[dot,color=red!60!blue!80] () at (\i*30+27:\ra) {};
        }
    \end{scope}

    \draw[->, line width = 0.4mm] (-2,0) -- node [above] {\Large $\Psi$} (1,0);

    \begin{scope}[shift = {(7,0)}, line width = 0.25mm]
        \filldraw[color=black,fill=yellow!20, line width = 0.2mm] (0,0) circle (\ra);

        \node[dot,color=red, inner sep = 2.5pt] () at (36:\ra) {};
        \node () at (36:\ra + 6mm) {$(4\ 5)$};

        \node[dot,color=black!30!green, inner sep = 2.5pt] () at (144:\ra) {};
        \node () at (146:\ra + 8mm) {$(1\ 2\ 3\ 11)$};

        \node[dot,color=blue, inner sep = 2.5pt] () at (180:\ra) {};
        \node () at (180:\ra + 13mm) {$(3\ 6\ 8)(11\ 12)$};

        \node[dot,color=red!60!blue!80, inner sep = 2.5pt] () at (324:\ra) {};
        \node () at (326:\ra + 8mm) {$(8\ 9\ 10)$};
        
    \end{scope}

\end{tikzpicture}

  \caption{A continuous noncrossing partition in $\NCS{12}{}$ and its
    corresponding weighted circular factorization in
    $\Fc(\sym_d,\delta,\Sb)$}
    \label{fig:degree-d-invariant-isom}
\end{figure}

\begin{thm}[Theorem~\ref{mainthm:nc-fact-isomorphism}]\label{thm:degree-d-invariant-isom}
  Define $\Psi \colon \mult(\sym_d,\Sb) \to \Pi_d(\Sb)$ by sending the
  $\sym_d$-multiset $0^{x_L}s_1^{x_1} \cdots s_k^{x_k}1^1$ to the
  partition $0^{\psi(x_L)} s_1^{\psi(x_1)} \cdots s_k^{\psi(x_k)}$.
  Then $\Psi$ restricts to an isomorphism from $\Fc(\sym_d,\delta,\Sb)$ to
  $\NCS{d}{}$.
\end{thm}

\begin{proof}
  Let $\bx = 0^{x_L}s_1^{x_1} \cdots s_k^{x_k}1^1 \in
  \mult(\sym_d,\Sb)$, define $\pi_j = \psi(x_j)$ for each $j$ and
  consider the partition $\Psi(\bx) = 0^{\pi_L} s_1^{\pi_1} \cdots
  s_k^{\pi_k}$. As outlined in
  Definition~\ref{def:degree-d-invariant}, $\Psi(\bx)$ is noncrossing
  if and only if the shuffle partition in $\Pi(dk+d)$ determined by
  $(\pi_L,\ldots,\pi_k)$ is noncrossing, and this is equivalent to
  having $x_L,\ldots,x_k \in [1,\delta]$ and $\ell(x_L) + \cdots +
  \ell(x_k) = \ell(x_L\cdots x_k)$ by
  Theorem~\ref{thm:shuffle-reduced-product}. By
  Lemma~\ref{lem:interval-factors}, this condition is satisfied by all
  elements of $\Fc(\delta,\Sb)$, so $\Psi$ restricts to a map
  $\Fc(\sym_d,\delta,\Sb) \to \NCS{d}{}$. See
  Figure~\ref{fig:degree-d-invariant-isom}.
	
  To see that $\Psi$ is surjective, let $\pi = 0^{\pi_L} s_1^{\pi_1}
  \cdots s_k^{\pi_k}$ be an element of $\NCS{d}{}$ and note that $k
  \leq d-1$ by Lemma~\ref{lem:total-criticality}. By
  Theorem~\ref{thm:nc-part-perm}, we may define $x_i =
  \psi^{-1}(\pi_i)$ for each $i$ and again apply
  Theorem~\ref{thm:shuffle-reduced-product} to see that
  $0^{x_L}s_1^{x_1} \cdots s_k^{x_k}1^1$ is an element of
  $\Fc(\delta,\Sb)$ which is sent to $\pi$. Injectivity of $\Psi$ then
  follows from Theorem~\ref{thm:nc-part-perm}.  And, by the definitions
  of the partial orders, $\bx \leq \by$ in $\Fc(\sym_d,\delta,\Sb)$ if
  and only if $\Psi(\bx)\leq \Psi(\by)$ in $\NCS{d}{}$. Therefore,
  $\Psi$ is an isomorphism and the proof is complete.
\end{proof}

\begin{figure}
  \centering
  \begin{tikzpicture}
  \def\ra{2cm}
  \def\mid#1#2{(#2-#1)/2}
  \def\Arc#1#2#3{arc ({#1+270}:{#2+90}:{tan(\mid{#1}{#2})*#3})}
  \def\drawArc#1#2#3#4{\draw[#1] (#2:#4) \Arc{#2}{#3}{#4}}

    \tikzstyle{multbackground} = [
        shape = rectangle, 
        draw = black,
        fill = yellow!40, 
        rounded corners, 
        minimum height = 4cm, 
        minimum width = 4cm
    ]
    
    \newcommand{\makemultiset}{
        \filldraw[color=black,fill=yellow!20, line width = 0.2mm] (0,0) circle (\ra);
    }
    
    \begin{scope}[x={(4cm,1cm)},y={(-.5cm,2cm)},z={(0cm,4cm)}]
        \pgfmathsetmacro{\z}{4};
        \pgfmathsetmacro{\t}{sqrt(3)};
        \coordinate (d0) at (0,0,0);
        \coordinate (a) at (1,\t,\z/2);
        \coordinate (b) at (-2,0,\z/2);
        \coordinate (c) at (1,-\t,\z/2);
        \coordinate (d1) at (0,0,\z);
    \end{scope}

    \begin{scope}[GreenPoly, line width = 0.4mm, opacity = 0.8]
        \filldraw (d0) -- (a) -- (d1) -- cycle;
        \filldraw (d0) -- (b) -- (d1) -- cycle;
        \filldraw (d0) -- (c) -- (d1) -- cycle;
    \end{scope}

    \node[dot] at (d0) {};
    \node[dot] at (a) {};
    \node[dot] at (b) {};
    \node[dot] at (c) {};
    \node[dot] at (d1) {};
    
    \node[dot] (e1) at (barycentric cs:d0=0.7,c=0.3,d1=0) {};
    \node[dot] (e2) at (barycentric cs:d0=0,a=0.7,d1=0.3) {};
    \node[dot] (e3) at (barycentric cs:d0=0.2,d1=0.8) {};
    \node[dot] (f1) at (barycentric cs:d0=0.4,b=0.4,d1=0.2) {};

    \coordinate (d0-label) at (0,-3);
    \coordinate (a-label) at (8,11);
    \coordinate (b-label) at (-11,6);
    \coordinate (c-label) at (9,5);

    \coordinate (d1-label) at (0,19);
    \coordinate (e1-label) at (6,0);
    \coordinate (e2-label) at (5,16);
    \coordinate (e3-label) at (-6,15);

    \coordinate (f1-label) at (-7,1);

    \foreach \x in {d0,a,b,c,d1,e1,e2,e3,f1} {
        \draw[->,shorten >=0.2cm, line width = 0.4pt] (\x-label) -- (\x);
    }

    \begin{scope}[line width = 0.5mm]
    \foreach \x in {d0,a,b,c,d1} {
        \begin{scope}[shift = {(\x-label)}]
            \makemultiset
            \filldraw[color=red, fill=red!20] (0:\ra)
            \Arc{0}{120}{\ra}
            \Arc{120}{240}{\ra}
            \Arc{240}{360}{\ra};
        \end{scope}
    }

    \begin{scope}[shift = {(e1-label)}]
        \makemultiset
        \draw[color=red] (120:\ra) \Arc{120}{240}{\ra};
        \draw[color=red] (280:\ra) \Arc{280}{400}{\ra};
    \end{scope}

    \begin{scope}[shift = {(e2-label)}]
        \makemultiset
        \draw[color=red] (120:\ra) \Arc{120}{240}{\ra};
        \draw[color=red] (280:\ra) \Arc{280}{400}{\ra};
    \end{scope}

    \begin{scope}[shift = {(e3-label)}]
        \makemultiset
        \filldraw[color=red, fill=red!20] (96:\ra) 
        \Arc{96}{216}{\ra} 
        \Arc{216}{336}{\ra}
        \Arc{336}{456}{\ra};
    \end{scope}

    \begin{scope}[shift = {(f1-label)}]
        \makemultiset
        \draw[color=red] (40:\ra) \Arc{40}{160}{\ra};
        \draw[color=red] (180:\ra) \Arc{180}{300}{\ra};
    \end{scope}
    \end{scope}

\end{tikzpicture}
  \caption{The isomorphism described in
    Theorem~\ref{thm:degree-d-invariant-isom} allows us to label the
    points in the dual braid complex from
    Figure~\ref{fig:weighted-fact-2-circle} by maximal elements of
    $\NCS{3}{}$. As before, points with the same label are glued
    together, which means that all five vertices are identified and
    the short edges are identified in pairs.}
    \label{fig:dual-braid-complex}
\end{figure}
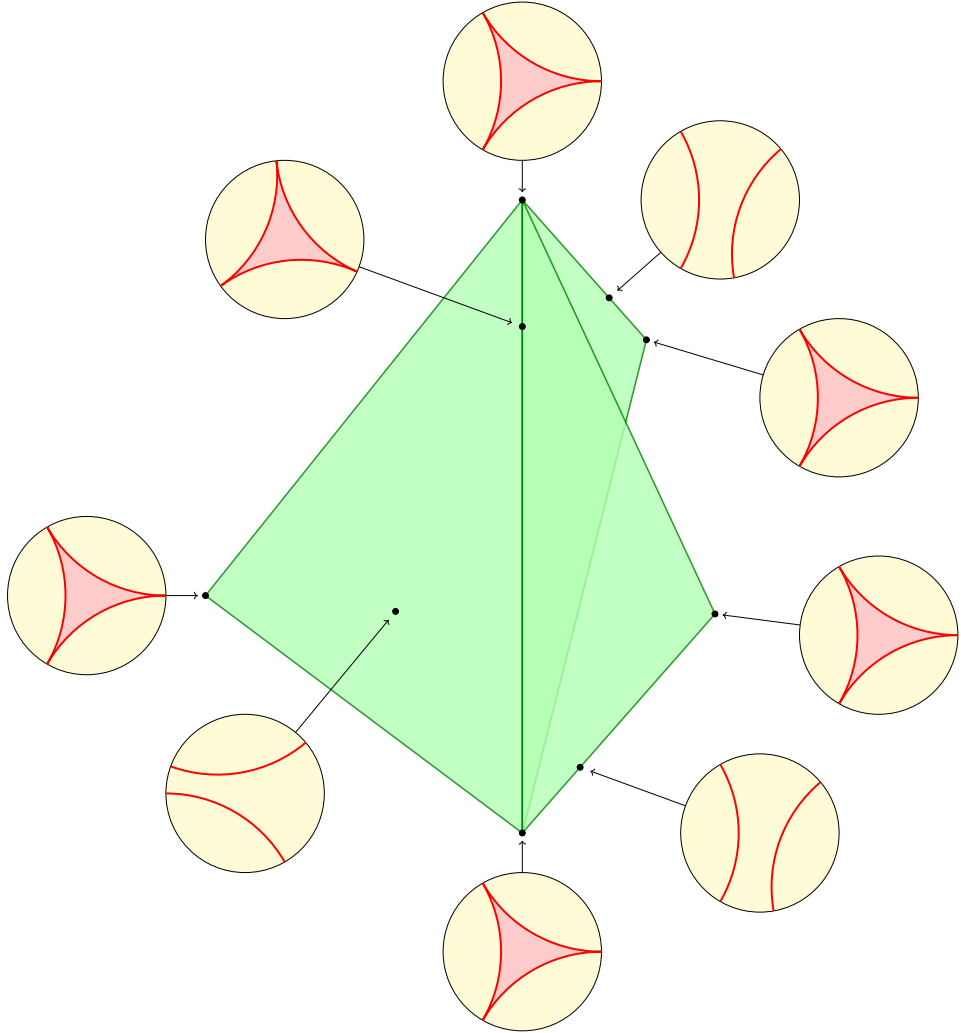

As a consequence of Theorem~\ref{thm:degree-d-invariant-isom}, we can
import the topology and cell structure from $\Fc(\sym_d,\delta,\Sb)$
to $\NCS{d}{}$---see Figure~\ref{fig:dual-braid-complex} for an
example using the maximal elements of $\NCS{3}{}$.  In
\cite{thurston20}, Thurston and his collaborators gave a topology for
the maximal elements of $\NCS{d}{}$ using a slightly different metric:
all edges of the dual braid complex would have equal length in their
metric, whereas the lengths differ in the orthoscheme metric as
described in Remark~\ref{rem:orthoscheme-faces}.  Nevertheless, the
two topologies are homeomorphic.

Let $\poly_d^{mc}(U)$ denote the space of monic degree-$d$ complex
polynomials for which the roots are centered at the origin and the
critical values lie in the subspace $U \subseteq \mathbb{C}$. The
following theorem from \cite{thurston20} uses the topology above to
provide a spine for $\poly_d^{mc}(\mathbb{C}^*)$, the space of
polynomials with distinct roots.

\begin{thm}[\protect{\cite[Theorem~9.2]{thurston20}}]\label{thm:thurston-def-ret}
  The space of maximal elements in $\NCS{d}{}$ is homeomorphic to
  $\poly_d^{mc}(\Sb)$, and there is a deformation retraction from
  $\poly_d^{mc}(\mathbb{C}^*)$ to $\poly_d^{mc}(\Sb)$.
\end{thm}

\begin{cor}[Theorem~\ref{mainthm:dual-braid-spine}]\label{cor:dual-braid-spine}
  The dual braid complex $K_\delta$ is homeomorphic to the subspace
  $\poly_d^{mc}(\Sb)$ of polynomials with critical values on the unit
  circle, and there is a deformation retraction from
  $\poly_d^{mc}(\mathbb{C}^*)$ to $\poly_d^{mc}(\Sb)$.
\end{cor}

\begin{proof}
    By Theorem~\ref{thm:thurston-def-ret},
    $\poly_d^{mc}(\mathbb{C}^*)$ deformation retracts to
    $\poly_d^{mc}(\Sb)$, which is homeomorphic to the maximal elements
    of $\NCS{d}{}$, which in turn is homeomorphic to the dual braid
    complex $K_\delta$ by Theorems~\ref{thm:max-elements}
    and~\ref{thm:degree-d-invariant-isom}.
\end{proof}

\bibliographystyle{alpha}
\bibliography{cnc}

\end{document}